\documentclass[twocolumn]{article}

\usepackage[accepted]{aistats2019}
\usepackage{amsthm}
\usepackage{bbm}
\usepackage{amssymb}
\newtheorem{theorem}{Theorem}
\newtheorem{lemma}{Lemma}
\newtheorem{fact}{Fact}
\newtheorem{proposition}{Proposition}
\newtheorem{corollary}{Corollary}

\newenvironment{proofs}{%
  \proof}{\endproof}

\usepackage{amsmath}               
  {
      \theoremstyle{plain}
      
  }
\RequirePackage{filecontents}
\usepackage{algorithm}
\usepackage[noend]{algpseudocode}
\makeatletter
\def\BState{\State\hskip-\ALG@thistlm}
\makeatother

\usepackage[noend]{algpseudocode}
\newcommand{\e}{\varepsilon}
\newcommand{\la}{\langle}
\newcommand{\ra}{\rangle}
\newcommand{\F}{\mathcal F}
\newcommand{\E}{\mathbb E}
\newcommand{\R}{\mathbb R}
\newcommand{\Var}{\text{Var}}

\newcommand{\de}{\delta}

\newcommand{\Prob}{\mathbb P}
\usepackage[utf8]{inputenc} 
\usepackage[T1]{fontenc}    
\usepackage{hyperref}       
\usepackage{url}            
\usepackage{booktabs}       
\usepackage{amsfonts}       
\usepackage{nicefrac}       
\usepackage{microtype}      
\usepackage{cleveref}
\usepackage{graphbox}
\usepackage{graphicx}
\usepackage{xcolor}
\usepackage{import}
\usepackage{soul}
\usepackage{multicol, blindtext}
\usepackage[font=small,labelfont=bf]{caption}
\usepackage{natbib}

\begin{document}

\twocolumn[
\aistatstitle{Safe Convex Learning under Uncertain Constraints}


\aistatsauthor{ Ilnura Usmanova \And  Andreas Krause \And  Maryam Kamgarpour }
\aistatsaddress{ Automatic Control Laboratory,\\
  ETH Zürich, Switzerland \And  Machine Learning Institute,\\
  ETH Zürich, Switzerland \And Automatic Control Laboratory,\\
  ETH Zürich, Switzerland
  } ]

 \begin{abstract}
We address the problem of minimizing a convex smooth function $f(x)$ over a compact polyhedral set $D$ given a stochastic zeroth-order constraint feedback model. This problem arises in safety-critical machine learning 
applications, such as personalized medicine and robotics. In such cases, one needs to ensure constraints  are satisfied while exploring the decision space to find optimum of the loss function. We propose a new variant of the Frank-Wolfe algorithm, which applies to the case of uncertain linear constraints. Using robust optimization, we provide the convergence rate of the algorithm while guaranteeing feasibility of all iterates, with high probability.\footnote{We thank the support of Swiss 
National Science Foundation, under the grant SNSF 200021\_172781, and ERC under the European Union's Horizon 2020 research and innovation programme 
grant agreement No 815943.}
\end{abstract}

\section{INTRODUCTION}

Many optimization tasks in robotics, health sciences, and finance require minimizing a loss function under  uncertainties. Most existing stochastic and online optimization approaches proposed to address these tasks assume that the constraints of the corresponding optimization problems are known.  
These approaches, however, are unacceptable in cases in which the feasible set is itself unknown and  is learned 
online.
Optimizing a loss function under such a partially revealed feasible set model is further challenged by the fact that exploration can be made only inside the feasible set due to safety reasons.  Hence, one needs to carefully choose actions to ensure feasibility of each iterate {with high probability} while learning the optimal solution. In the machine learning community, this problem is known as \textit{safe learning}. 

Safe learning is receiving increasing attention  due to the increasingly widespread deployment of machine learning 
in safety-critical tasks. An example 
arises in 
personalized medicine, where physicians may choose from a large set of therapies. The effects of different therapies on the patient are initially unknown, and can only be determined through clinical trials. 
Free exploration however is not possible since some therapies might cause discomfort or even physical harm \citep{pmlr-v37-sui15}. 
Similar challenges  arise in designing control algorithms for robots, which have to navigate  unexplored terrains or interact with humans \citep{cassandra1996acting,koenig1996unsupervised}. In these scenarios, robots need to learn the best tuning for their controllers or optimize their trajectories based on risky experimental interactions with the partially known environment. 

We address the problem of safe learning given a convex loss function subject to unknown constraints. Motivated by the aforementioned applications, we assume that the decision maker has access only to noisy observations of the constraints for a chosen action. Our objective is to design an algorithm that sequentially steers the decisions towards the optimum while ensuring safety of the decisions at every step. In this paper, we restrict ourselves to unknown linear constraints. Perhaps surprisingly, there is very little work on safe learning for this rudimentary setup. Hence, we consider this as a first step towards developing fundamental understanding and design of efficient algorithms for the more general nonlinear and non-convex setting. 

\paragraph{Related work.}

There are many optimization  algorithms that ensure feasibility of the iterates, assuming a known constraint function.
The most basic ones are projected gradient descent (PGD)\citep{boyd2004convex} and Frank-Wolfe (FW) \citep{frank1956algorithm} (also known as conditional gradient).   
These methods require exact knowledge of the constraints or at least a projection oracle or an exact linear programming (LP) oracle with respect to {the} constraints.  
{However,} as discussed above, this information may not be available in safe learning problems.

Assuming the functional form of the constraints, with parameters  drawn from a probability distribution, chance-constraint optimization addresses the problem of optimizing a loss function subject to constraint satisfaction with a sufficiently high probability  \citep{el1998robust}.
The proposed solution methods for this problem assume either a priori knowledge of the distribution of the parameters \citep{el1998robust}
, moments of the distribution \citep{zymler2013distributionally}, 
or a sufficiently large number of samples of the distribution \citep{calafiore2005uncertain}. In contrast, in a safe learning  problem, the decision maker does not have access to such information a priori. This information is gathered {online} and feasibility needs too be ensured while exploring the uncertain decision space. 

A recent line of work addresses uncertain constraints in online stochastic optimization \citep{yu2017online,yu2016low}. The work is based on infeasible penalty methods, and thus does not provide guarantees on constraint violation at each iteration. Rather, the methods ensure 
convergence of the \emph{average} constraint violation to zero. Similarly, risk-aware contextual bandits  and bandits with knapsack constraints 
\citep{sun2017safety,mahdavi2012trading,jenatton2015adaptive} consider unknown constraint functions with a budget limit. Here, safety refers to ensuring that the total usage of a commodity, e.g.,  budget for adverts,  summed over the sequence of iterates remains below a threshold. Similar to \cite{yu2017online,yu2016low}, the above approaches bound average constraint violation, rather than avoiding violation at each iteration. While such a formulation can be well-suited in certain problems such as adverts, it may not be well-suited for  safe learning applications discussed above because in  this latter case, constraints need to be satisfied at each step. 

The problem of safe learning using Gaussian processes (GP) has been proposed in \cite{pmlr-v37-sui15}. The  SafeOpt algorithm developed in the above work considers minimizing an unknown loss function iteratively, while ensuring that the loss of each iterate is above a required threshold. Given  {actively taken} measurements of the loss, the initial estimate of the feasible set is incrementally enlarged through exploration and considering certain regularities of GP kernels. This framework is extended to multiple constraints and experimentally validated on robotic platforms by \cite{berkenkamp2016bayesian}. Safe GP learning is powerful as it can address general non-convex problems.  
Nevertheless, due to this generality, current approaches do not scale well with the problem dimension. This motivates our work of developing efficient safe learning algorithms for the case of convex loss functions and constraints. 

\paragraph{Our contributions.} We propose a novel algorithm for safe 
{active} learning, given a smooth convex objective and a set of unknown linear constraints with noisy oracle information.
Given a confidence level $1-\delta$ and accuracy $\epsilon$,  we prove that after $\tilde O\left(\frac{1}{\epsilon}\right)$ iterations and $\tilde O\left(\frac{d^3\ln 1/\delta}{\epsilon^2}\right)$ constraints measurements, the final point is an $\epsilon$-accurate solution with probability $1-\delta$ (cf. Theorem 2). 
By $\tilde O(\cdot)$ we denote $O(\cdot)$ up to a logarithmic multiplicative factor. 
Furthermore, we ensure feasibility 
for the trajectory of the iterates with probability at least $1-\delta$ (cf. Theorem 1). 
While in this paper we mainly focus on exact first-order oracles for the objective function, we discuss extensions to stochastic oracles 
in \Cref{sec:convergence}. 

The core idea of our algorithm is to combine a first-order feasible optimization approach with the robust optimization technique. Specifically, our algorithm is based on the Frank-Wolfe (FW) method. 
In each iteration, it solves an \emph{uncertain} linear program based on  estimates of the constraints and uses this solution to define the step direction. The safety of the iterates is ensured with high probability by refining the confidence set of the unknown parameters iteratively. We emphasize that while we use robust optimization 
\citep{ben1998robust, ben1999robust, ben2000robust},
our problem formulation is different than that of a classical robust optimization. Specifically, we consider gathering information {online} about the uncertainty, whereas the robust optimization works assume one-shot knowledge of uncertainties.
We numerically evaluate the performance of the proposed algorithm in Section \ref{sec:simulations} and compare its performance  with a one-shot robust optimization approach.

\paragraph{Notations}
Let $I^d \in \R^{d\times d}$ 
denote identity matrix,  $e_i\in \R^d$ the unit  vector corresponding to the $i$-th coordinate. Let $\|\cdot\|$ denote the Euclidean norm for vectors and the spectral norm for matrices. The ball of radius $r$ centered at a point $x_0\in \R^d$ is denoted by $B^d(x_0,r) = \{x\in \R^d| \|x-x_0\|\leq r\}$. 

\section{PROBLEM FORMULATION}\label{sec:problem}
The problem of safe learning in its most general form can be defined as a constrained optimization problem 
\begin{align*}
  &\min_{x\in \R^d} ~~~~~~~f(x) \\
 &\text{subject to }  f_i(x) \leq 0~~ \forall i=1,\ldots,m,
\end{align*}
where the objective function $f: \R^d \rightarrow \R$ and the constraints $f_i: \R^d \rightarrow \R$ are unknown, and can only be accessed at feasible points $x$. A possibly noisy oracle provides access to the values of these  functions or their gradients at any queried feasible $x \in \R^d$. The objective is to design an iterative algorithm that chooses the query points to ensure feasibility at each round while progressing towards the optimum. To address this goal, we need to define the oracle more precisely and make some regularity assumptions on the functions.

In this paper, we consider an instance of the safe learning problem in which the objective $f$ is convex and $M$-Lipschitz continuous, that is,  $|f(x) - f(y)| \leq M\|x-y\| ~ \forall x,y \in D$, where $D$ is the feasible set. Furthermore, we assume $f$ is $L$-smooth, that is, $f$ has $L$-Lipschitz continuous gradients in $D$,  $\|\nabla f(x) - \nabla f(y)\| \leq L\|x-y\|,~ \forall x,y \in D$. {We assume access to the gradients of the objective function, $\nabla f(x)$,  at any feasible query point $x \in D$.} We furthermore assume that constraints are known to be linear, $f_i(x)  = [a^i]^Tx - b^i$ for $i=1, \dots, m$. Hence, letting $A \in \R^{m \times d}$ denote the matrix with rows defined by $[a^i]^T$, the problem is given by 
\begin{align}\label{problem}
&\min_{x\in \R^d} ~~~~~~~~~f(x)\\
 &\text{subject to }\, Ax-b \leq 0. \nonumber
\end{align}
 We assume that the feasible set $D = \{x\in \R^d: Ax-b\leq 0\}$ is a compact polytope  with non-empty interior. Denote by $\Gamma$ the diameter of the set $D$, $\Gamma =  \max_{x,y\in D}\|x-y\|$. Furthermore, let $\Gamma_0$ be the radius of the smallest ball centered at 0 such that $D\in B(0,\Gamma_0)$, 
namely,  $\Gamma_0 =  \max_{x\in D}\|x\|$.

If $A$ and $b$ are known, (\ref{problem}) can be solved efficiently by off-the-shelf first-order convex optimization algorithms. We however, consider the case in which $A$ and $b$ are unknown and can be accessed through an oracle. Specifically, we assume the constraints can be evaluated at any point that lies within a ball of radius $\omega_o$ of the feasible set. These evaluations are corrupted by sub-Gaussian noise.
 Hence, we have access to  $y(x)= Ax - b + \eta$  for any $x$ such that  $B^d(x,\omega_0)\cap D \neq \emptyset$, where $\eta$ are sub-Gaussian. 
 {If in the problem setting having all the measurements inside the feasible set is critical, we can artificially shrink the set $D$ by the value $\omega_0$ from the boundaries. This can be achieved by tightening the constraints $[a^i]^Tx\leq b^i$ with setting the measurements $\hat y^i = y^i - \mathcal\kappa = [a^i]^Tx - b^i + \eta - \mathcal\kappa^i$, with $\mathcal\kappa^i \geq L^i_{A}\omega_0$, where $L^i_{A}$ is an upper bound on $\|a^i\|$. 

The scope of the present paper is to design an algorithm which, starting from a feasible point $x_0 \in D$, converges to an optimal solution $x_*$ with a required accuracy $\epsilon$ and a required confidence $1-\de$ after $T$ steps, that is, \begin{align}\Prob \{ f(x_{T}) - f(x_*) \leq  {\epsilon}\} \geq 1-\de.\end{align}  Since the constraint set $D$ is unknown and revealed through a noisy oracle, we can at the very best  ensure to  remain inside the feasible set with sufficiently high probability. Hence, we require that the updates of the method are not violating the true constraints with the same required confidence level of $1-\de$, that is, 
\begin{align}
\label{eq:safety_iterate}
\Prob\{ Ax_t -b \leq 0, \; 0\leq t\leq T\} \geq 1-\de.
\end{align} 

Some words on the choices of the optimization and oracle above are in order. First, the setting of linear constraints can be restrictive for some real-world problems. Nevertheless, understanding the linear setup is often the first step in addressing more challenging formulations. Second, having a noisy first-order or  a zeroth-order oracle for the objective function is more realistic for several safe learning problems. Optimization under such oracle models have been deeply explored for the case in which the constraint set is known. Hence, the main novelty and challenge in  safe learning  is ensuring feasibility of the iterates despite uncertain and incrementally revealed constraint values. We discuss how the proposed algorithm can be generalized to stochastic oracle models for objective in \Cref{sec:convergence}.

\section{THE SFW ALGORITHM}\label{sec:algo}

We propose a variant of the Frank-Wolfe algorithm where we explicitly take into account the uncertainty about the feasible set $D$, referred to as Safe Frank-Wolfe (SFW). 
The algorithm can be summarized as follows. Starting with a feasible point $x_0\in D$, at each iteration $t =0, \dots, T$ we generate a number $n_t$ of query points 
and obtain noisy measurements of the constraint functions at these points. 
Using linear regression, we obtain an estimate $\hat{D}_t$ of the feasible set 
based on the history of obtained measurements. 
The algorithm then {uses }
$\hat{D}_t$ to obtain a direction $\hat{s}_t$ by solving the estimated Direction Finding Subproblem (DFS)  
\begin{align}\label{dfs}
\hat{s}_t = \arg\min_{s\in\hat{D}_t} \la \nabla f(x_t),s \ra. \end{align}
The next iterate is then given by $x_{t+1} = x_t + \gamma_t (\hat s_t - x_t)$, according to a chosen step-size $\gamma_t$. 
Below, we further describe each step of the proposed algorithm.

\paragraph{Taking Measurements.} During each iteration $t$ of the algorithm, we first make measurements at 
$n_t$ number of points $x_{(j)}$ within distance $\omega_0$ of $x_t$ in $d$ linearly independent directions. The  number $n_t$ needs to satisfy a lower bound  as a function of the input data $\de$, $T$, to ensure safety. This bound is provided in \Cref{thm:safety}.
Denote by $X_t = [x_{(1)},\ldots,x_{(N_t)}]^T \in \R^{N_t\times d}$  and by $N_t = \sum_{k=0}^t n_k$, the total number of available measurements at iteration $t$.  
Combining  all measurements taken up to iteration $t$ we have the following information about the constraints 
$y^i = X_t a^i -  b^i\textbf{1} + \eta^i, \,i = 1,\ldots,m,$
where $y^i \in \R^{N_t}$ is the vector of $N_t$ measurements of $i$-th constraint, $\eta^i= [\eta^i_{1},\ldots,\eta^i_{N_t}]^T\in \R^{N_t}$ is the vector of errors.  
The errors $\eta^i_j$ are 
independent and $\sigma$-sub-Gaussian, which means 
$$\forall \lambda \in \R ~\forall i = 1,\ldots,m~\forall  j \leq N_t~ \E\left[ e^{\lambda \eta^i_{j}}
\right] \leq \text{exp}\left(\frac{\lambda^2\sigma^2}{2}\right).$$ The sub-Gaussian condition implies that $\E\left[\eta^i_{j} 
\right] = 0$ and $\Var\left[\eta^i_{j} 
\right] \leq \sigma^2$. 
An example of $\sigma$-sub-Gaussian $\eta^i$ are independent zero-mean Gaussian random variables with variance at most $\sigma^2$, or independent bounded zero-mean variables lying in an interval of length at most $2\sigma$.  We also denote by $\mathbf{1} \in \R^{N_t}$ the vector of 1's. 
Let us denote by $Y_t = [y^1,\ldots,y^m]\in \R^{N_t\times m}$ 
the matrix of corresponding measurements of the constraints.

\paragraph{Estimating  Constraints. }
Let $\beta^i = \begin{bmatrix}[a^i]^T&
b^i\end{bmatrix}^T\in \R^{d+1}$ denote  the vector corresponding to the $i$-th constraint. We refer to $\beta^i$ as the true parameter. Let $\bar{X}_t = [X_t,-\textbf{1}]\in \R^{N_t \times (d+1)}$ be the extended version of the matrix $X_t$. The Least Squares Estimation (LSE)  of the constraint parameters at step $t$ is given by \begin{align}\label{mse}
\hat{\beta}_t = [\hat{A}_t, \hat{b}_t]^T = [\bar{X}_t^T\bar{X}_t]^{-1}\bar{X}_t^TY_t. \end{align}
The covariance matrix of the $\hat \beta^i_t$ is given by $\Sigma_t = \sigma^2[\bar{X}_t^T\bar{X}_t]^{-1}$. 
Let $\hat a_t^i, \hat b_t^i$ denote the estimates of the corresponding rows of  $\hat \beta^i_t$ and  $\hat{D}_t = \{x\in\R^d :\hat{A}_t x\leq \hat{b}_t\}$ denote the estimated feasible set.

\paragraph{Stopping criteria. }
Recall that $\hat{s}_t$ is the minimizer of the estimated DFS (\ref{dfs}) and let $\hat{g}_t$  {be }
{its }
optimal value 
%
\begin{align}\label{gap}
\hat g_t &= \min_{s\in \hat D_t}\la  s,\nabla f(x_t)\ra.
 \end{align} 
Similarly, let $s_t$ denote the minimizer of the DFS under true constraints and ${g}_t$ the corresponding optimal value
\begin{align}\label{sol_gap}
s_t= \arg\min_{s\in D }\la s,\nabla f(x_t)\ra, ~~
g_t = \min_{s\in D}\la s, \nabla f(x_t)\ra.
\end{align}
From convexity of $f$, we have  $f(x_t) - f(x_*) \leq g_t$. Thus, as discussed in \cite{jaggi2013revisiting}, $g_t$ can be taken as a surrogate duality gap and consequently a stopping criterion for the FW algorithm. In our case, the duality gap cannot be computed exactly because the feasible set $D$ is unknown. 
{Nevertheless, for the random variable $E_t:= |\hat{g}_t - g_t|$ describing an error in the gap estimation we can derive a probabilistic upper bound $\bar E_t(\de)$, such that $\Prob\{
{E_t}\leq \bar E_t(\de)\} \geq 1-\de$ (see \Cref{prop:dfs} \Cref{sec:convergence}). It follows that if $\hat{g}_t +  \bar E_t(\de) \leq \epsilon$, then with probability greater than $1- \de$ we have
$f(x_t) - f(x_*) \leq  \epsilon$.}
Thus, we use $\hat{g}_t + \bar E_t( \de)\leq \epsilon$ as a stopping criterion.

Putting the above few steps together, we present the Safe Frank-Wolfe (SFW) in Algorithm \ref{scfw}. 
\begin{algorithm}[h]
\caption{SFW (Safe Frank-Wolfe)}\label{scfw}
\begin{algorithmic}[1]
 \BState \emph{Input:} $x_0 \in D$, bound on iterations $T$, accuracy $\epsilon$,  confidence parameter {$\de$}, measurement radius $\omega_0$;
 \State $t \gets {0}$; Choose $n_t(\de, T)$;
 \While {$t \leq T$}
\State  {Pick $2d$ points around 
the current point $x_t$  
\begin{align*}
&x_{(N_{t-1}+i)} = x_t  + e_i\omega_0,\\
&x_{(N_{t-1}+2i)} =  x_t  - e_i\omega_0, i=1,\ldots,d, 
\end{align*}
and take $[n_t/2d]$ measurements at each point;}

\State Obtain the gradient $\nabla f(x_t)$ and the noisy constraint values $y^i_{(j)} 
= x_{(j)}^T a^i -  b^i + \eta^i_{(j)}$ {$\forall j =N_t+1,\ldots, N_t+n_t, i =1,\ldots, m$};
\State Compute the LSE of the constraints $\hat{A}_t$ and $\hat{b}_t$ based on (\ref{mse});
\State Solve the  estimated DFS (\ref{dfs}) to obtain $\hat s_t$;
\State 
 Estimate  {the} duality gap $\hat{g}_t$ (\ref{gap});
 \If {$\hat{g}_t \leq \epsilon - E_t(\bar \de)$} break \textbf{and} return $x_t$; \EndIf
 \State Set $\gamma_t = \frac{1}{t+2}$; 
\State  $x_{t+1} \gets x_t + \gamma_t (\hat{s}_t - x_t);$
\State $t \gets t+1.$
\EndWhile
\end{algorithmic}
\end{algorithm}


\section{SAFETY}\label{sec:safety}
In order to ensure safety of the trajectory as per Inequality \eqref{eq:safety_iterate} we ensure that each $x_{t+1}$ generated by the algorithm above
remains within the feasible set $D$ with probability $1-\bar \de,$
where $\bar \de  = \frac{\delta}{T}$.
This is achieved using the analysis framework of robust optimization by \cite{bertsimas2011theory}, \cite{ben1998robust}. The safety of each iterate,  combined with a union bound, enables us to prove the safety of the sequence $\{x_t\}_{t=1}^{T}$ with probability $1-\delta = 1- \sum_{t=1}^{T} \bar \de$.
  
Since the LSE's of the constraint parameters are given by  $\hat{\beta}_t^i = [[\hat{a}^i_t]^T, \hat{b}_t^i]^T\in\R^{d+1}$,
the confidence set for the vector of true parameters $\beta^i$ is given by the following ellipsoid:
$\mathcal{E}^i_t(\bar \de) = \left\{z \in R^{d+1} : (\hat{\beta}_t^i - z)^T \Sigma_t^{-1}(\hat{\beta}_t^i - z)\leq \phi^{-1}
(\bar \de)^2\right\}, $
where, 
$$\phi^{-1}(\bar \de) =  \max\left\{
\sqrt{128 d \log N_t \log \left(\frac{N_t^2}{\bar\delta}\right)},
\frac{8}{3}\log \frac{N_t^2}{\bar\delta}\right\}$$ for $\sigma$-sub-Gaussian noise $\eta^i$ for $N_t e^{-1/16} \geq \bar\delta$, \citep{dani2008stochastic}
\footnote{In the case when the noise is Gaussian and $X_t$ is chosen deterministically, e.g. if all the samples are taken around $x_0$, 
 $\phi^{-1}(\bar \de)$ is taken as the inverse of Chi-squared cumulative distribution function with $d+1$ degrees of freedom \citep{draper2014applied}.}.  
Thus, we have an ellipsoidal uncertainty set centered at $\hat{\beta}_t^i$, such that $\Prob\{\beta^i\in \mathcal{E}^i_t(\bar \de)\} \geq 1-\bar \de.$
We define the confidence set $\mathcal{E}_t(\bar \de)$ for all parameters $\beta$  by $\mathcal{E}_t(\bar \de) = \mathcal{E}^1_t(\bar \de/m)\times \ldots \times \mathcal{E}^{m}_t(\bar \de/m)\subseteq \R^{(d+1)m}.$ 
The confidence set $\mathcal{E}_t(\bar \de)$ determines the uncertainty set for constraint parameters $\beta$ with probability $1-\bar \de$. Indeed, $
1- \Prob(\exists i:\, \beta^i \notin \mathcal E^i_t(\bar \de/m)) = 1-\Prob\{\cup_{i=1}^{m} \{\beta^i \notin \mathcal E^i_t(\bar \de/m)\}\} \geq 1-
\sum_{i=1}^{m}\bar\de/m=1-\bar\de.$

We define the safety set $S_t(\bar \de) \subset \R^d$ at iteration $t$ as the set of $x\in\R^d$ satisfying the constraints with any true parameter $\beta = [A,b]$ in the confidence set: 
\begin{align}
\label{eq:safety_set}
S_t(\bar \de) = \{x\in \R^d: Ax\leq b, ~ \forall [A,b]\in \mathcal{E}_t(\bar \de)\}.
\end{align}
Given that for each constraint $\beta^i$ our uncertainty set $\mathcal{E}^i_t(\bar \de)$ has an ellipsoidal form, the safety set $S_t(\bar \de)$ is  equivalent to the intersection of a  set of second order cone constraints \citep{ben2009robust} 
\begin{align}\label{INEQ0}
 S_t(\bar \de) = \Bigg\{ x\in &\R^d
 : \forall i=1,\ldots,m~\left[[\hat{a}^i_t]^T x - \hat{b}_t^i\right]+\nonumber\\&
+\phi^{-1}(\bar \de/m) \left\|
{\Sigma_t^{1/2}}\begin{bmatrix}x\\-1\end{bmatrix}\right\| \leq 0 \Bigg\}.
\end{align}
\begin{fact}\label{lemm:safety}
From the definition of the confidence and safety sets it readily follows that
$$\Prob \{x \in D\,|\, x \in S_{t}(\bar \de), \beta \in \mathcal E_{t}(\bar \de)\}  = 1.$$
\end{fact}

\begin{fact}~\label{fact:fact1}
The condition $x_t\in S_{t}(\bar\de)$ is equivalent to  \begin{align*}
 \phi_{\bar \de}\sqrt{\frac{1}{N_t}+ 
 (x_t - \bar{x}_t)^TR_t(x_t - \bar{x}_t)} \leq \min_{i=1,\ldots,m}\e_t^i, 
\end{align*}
where $\phi_{\bar \de} = \sigma\phi^{-1}(\bar \de/m)$,  
$\e^i_{t} = \hat{b}_t^i-[\hat{a}_t^i]^Tx_{t}$, $\bar x_t = \frac{X_t^T\mathbf{1} }{N_t}$, and 
$R_t = \big[\sum_{j=1}^{N_t}(x_{ {(j)}} - \bar{x}_t)(x_{(j)} - \bar{x}_t)^T\big]^{-1}.$
\end{fact}
The proof of this fact is provided in \Cref{proof:fact2}.


To state the main result on safety of each iteration, we need to introduce some notation. For the polytope $D\in \R^d$, by an active set $B$ we denote a set 
of indices of 
$d$ linearly independent constraints active in a vertex $V\in\R^d$ of $D$, i.e., $V = V^B = [A^B]^{-1}b^B$. Here, $A^B$ is a corresponding sub-matrix of $A$ and $b^B$ is the corresponding right-hand-side of the constraint. Let $\rho_{\min}(A^B)$ denote the smallest singular value of $A^B$.  Let  $Act(D)$ denote the set of all active sets 
corresponding to vertices of $D$, i.e., $Act(D) = \{B: V^{B} \text{ is a vertex of }D\}$. 
Furthermore, define $\rho_{\min}(D) := \min\{\rho_{\min}(A^B): B\in Act(D)\}$. Note that $\rho_{\min}(D) > 0$ since by definition, $B$ is a set of linearly independent active constraints. 
Let $\e_0 = 
\min_i\{b^i - [a^i]^T x_0\}$, and  $L_A = \max_i\|a^i\|$. 
{With} the notation in place, we can present the following lemma on the lower bound on the number of measurements to ensure safety of each iterate.
\begin{lemma}\label{prop:gamma}
If $\beta\in \mathcal E_{k}(\bar \de)$ for $k= 1,\ldots,t$   
 and $n_t = 4C_n t(\ln t)^2$, with the constant parameter $C_n$ satisfying
\begin{align}
C_n&\geq
 C_{\bar \de}^2\max\left\{
\frac{4 (\ln \ln T)^2L_{A}^2 }{[\e_0]^2},\frac{1}{(\Gamma_0 +1)^2}\right\}, \label{def:cn}
\end{align}
\begin{align}
\text{where } C_{\bar \de} &= \frac{2 \phi_{\bar \de} d (\Gamma_0+1) }{\rho_{\min}(D)}\sqrt{\frac{\Gamma_0^2+1}{\omega_0^2} + 1},\label{def:c}
 \end{align}
then $x_{t}\in S_{t}(\bar \de).$ Furthermore, the total number of measurements then satisfies  $N_t = C_n t^2(\ln t)^2$.
\end{lemma}
We provide the full proof in \Cref{proof:lemma1}.
\begin{proofs}Let us give a brief intuition for the proof. 
First, from Fact \ref{fact:fact1} we see that in order to have $x_t \in S_t(\bar\de)$ we require $\min_i\e_t^i \geq \Omega\left( \frac{1}{\sqrt{N_t}}\right)$, where $\e_t^i$ 
 is equal to the distance 
to the boundary   corresponding to the estimated $i$-th constraint  multiplied by $\|\hat a_t^i\|$. 
Second, if these estimates were fixed, then using step sizes $\gamma_t = \frac{1}{t+2}$ we could ensure that the convergence to any boundary $i$ would not be faster than $\e_t^i \geq \prod_{k=0}^t(1-\gamma_k)\e_0^i = \frac{\e_0^i}{t+2}$ (see \Cref{illustration}). 
Hence, we require $N_t \geq \Omega\left(\frac{t^2}{\e_0^2}\right)$. 
However, since  $\e_t^i$ are random and boundaries are fluctuating, we need $N_t$ to be square logarithmic times larger than the above estimate (as shown in the full proof).
\end{proofs}

\begin{figure}[!t]
\centering
\includegraphics[align=c,width=0.48\textwidth]{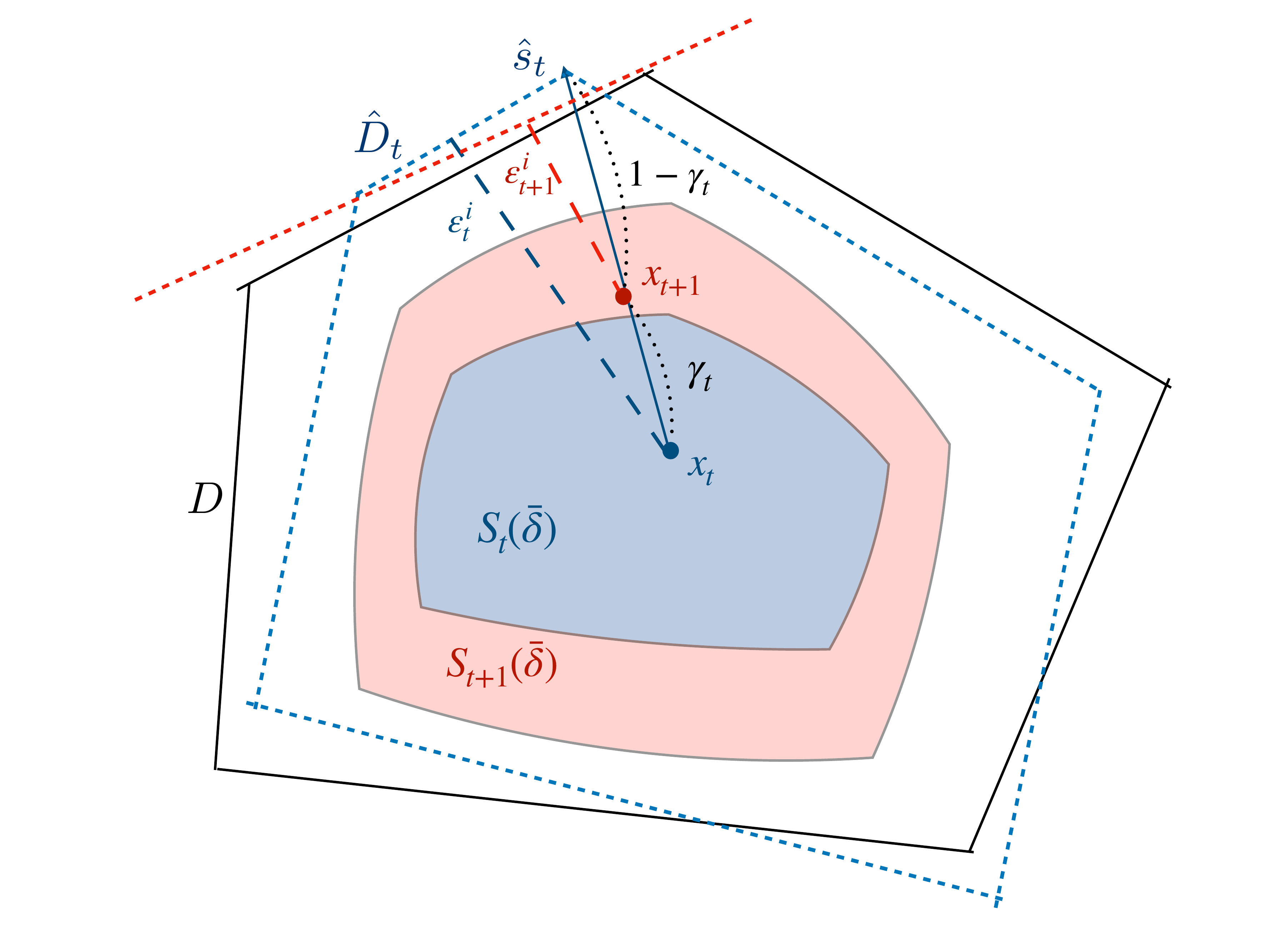}
\caption{Illustration of one iteration of SFW $x_{t+1} = x_t + \gamma_t (\hat s_t - x_t)$. Bold lines  denote the  polytope $D$ and  dashed lines denote its estimate $\hat D_t$. 
}
\label{illustration}
\end{figure}
\textbf{Remark.} Note that  the dependence on $\ln \ln T$ is very mild because 
this term grows extremely slowly, i.e,  $\ln \ln 15 \approx 1$ and $\ln \ln 2000 \approx 2$. 

Having established \Cref{prop:gamma}, we can present the safety guarantee of the SFW algorithm. 
 \begin{theorem}\label{thm:safety}
If $n_t = 4C_n t(\ln t)^2$, where the constant parameter $C_n$ is  defined in (\ref{def:cn})
 then, the sequence of iterates $\{x_t\}_{t=0}^{T}$ of SFW is feasible with probability at least $1-\delta$, that is,
  $\Prob\{x_t\in D \text{ for all } 0 \leq t \leq T\} \geq 1-\de.$
 \end{theorem}

\begin{proof}
Let $\mathcal{F}_t$ denote
the event that all the estimated confidence sets $\mathcal E_k(\bar \de)$ up to step $t$ cover  $\beta$, i.e., $\mathcal{F}_t = \{\beta \in \cap_{k=0}^{t} \mathcal E_{k}(\bar \de)\}.$
Furthermore, let $\mathcal{Q}_t$ denote the event that all the $x_k \in S_k(\bar \de)$ up to iteration $t$, i.e, 
$\mathcal{Q}_t = \{x_k \in S_k(\bar \de) \text{ for all } 0 \leq k \leq t\}.$
By \Cref{prop:gamma} if $\F_t$ holds, then $x_t\in S_t(\bar \de).$
Hence, it is easy to see that $\F_t$ implies $\mathcal Q_t$, i.e., $\Prob \{\mathcal Q_t|\F_t\} = 1.$ Thus, using Fact 1 we derive
\begin{align*}
    \Prob &\left\{x_{t} \in D \text{ for all } 0 \leq t \leq T\right\} =\Prob \{ \mathcal{Q}_{T},\mathcal{F}_{T}\}  =
    \Prob\{\mathcal{F}_{T}\}. 
\end{align*}
Using Boole's inequality we can bound the probability of $\F_T$ as follows
\begin{align*}
\Prob\{\mathcal{F}_{T}\} &= 
1-\Prob\{\cup_{t=0}^{T}\cup_{i=1}^{m} \{\beta^i \notin \mathcal E^i_t(\bar \de/m)\}\}\\& \geq 1-
\sum_{t=0}^{T} \sum_{i=1}^{m}\bar \de/m\geq 1-T \bar \de.
\end{align*} This concludes the proof.
\end{proof}

\section{CONVERGENCE}\label{sec:convergence}
First, we show that the proposed algorithm achieves the optimal convergence rate for the Frank-Wolfe algorithm (see \cite{lan2013complexity})
, with sufficiently high probability. Second, we discuss extensions to stochastic first-order and zeroth-order oracles  of the objective function based on the results of \cite{hazan2016variance}.
\subsection{Convergence rate.}
Let us define the curvature constant $C_f$ of the function $f(x)$ with respect to the compact domain $D$ by $$C_f = \sup_{ \substack{x,s \in D,
 \gamma \in [0,1],\\  y = x+\gamma(s-x)}} \frac{1}{\gamma^2} (f(y) - f(x) - \la y-x, \nabla f(x)\ra). $$
It can be verified that $C_f \leq L\Gamma^2$, where $L$ is the 
Lipschitz constant of the gradient $\nabla f(x)$ and $\Gamma$ is the diameter of the set $D$ (see \Cref{sec:problem}). Our main result is as follows.
\begin{theorem}\label{thm:conv}
If $n_t = 4C_n {(t+2)(\ln (t+2)})^2$ and $C_n$ is chosen according to the bound in Equation (\ref{def:cn}), then:\\
a) after $T$ steps of the SFW algorithm, the final point $x_T$ satisfies
\begin{align*}\Prob \Bigg\{&f(x_T) - f(x_*)  \leq \frac{f(x_0) - f(x_*)}{T+2}+\\& +\frac{\ln {(T+2)} \frac{C_f}{2} +\ln \ln (T+2)\frac{C'}{2}}{T+2} 
\Bigg\} \geq 1 - \de,\end{align*}
where 
$C' = \frac{M C_{\bar \de}}{\sqrt{C_n}} $, and $C_{\bar \de}$ is defined in (\ref{def:c}). \\
b) all the iterates $\{x_t\}_{t=1}^T$ are feasible with probability $1-\de$, as required in (\ref{eq:safety_iterate}).
\end{theorem}

\begin{corollary}\label{cor:conv}
The SFW algorithm achieves an $\epsilon$-accurate solution with probability greater than $ 1 - \delta$ after making   $\tilde O(\frac{1}{\epsilon})$ linear optimization oracle calls
 and $\tilde O\left(\max\left\{\frac{d^3 \ln\frac{1}{\bar \de}}{\epsilon^2}, \frac{d^2 \ln^2\frac{1}{\bar \de}}{\epsilon^2}\right\}\right)$ zeroth-order inexact constraint oracle  calls. \end{corollary}

Below, we provide the proof sketch for \Cref{thm:conv}. The full proofs of \Cref{thm:conv} and \Cref{cor:conv} are provided in \Cref{proof:thm2}. 
\begin{proofs}

Our proof is based on the extensive study of FW convergence provided by \cite{jaggi2013revisiting}, \cite{freund2016new}. 
Recall that $E_t$ is the accuracy with which an approximated DFS at iteration $t$ is solved. 
Similarly to (\cite{freund2016new}, Theorem 5.1)  we can show that for $\gamma_t = O\left(\frac{1}{t}\right)$, we have 
\begin{align}\label{freund}
f(x_t)-f(x_*)\leq O\left( \frac{\epsilon_0 + C_f\ln t 
+ \sum_{t=1}^{T} E_t}{t}\right),
\end{align}
where $\epsilon_0 = f(x_0) - f(x_*)$.
Hence, to prove the convergence rate of the SFW algorithm we need to show  that the error in the DFS solution decreases with the rate  $O\left(\frac{1}{t}\right)$. This fact is shown in \Cref{prop:dfs} below. 


\begin{proposition}\label{prop:dfs}
If $\beta \in \mathcal E_t(\bar \de)$ and $N_t \geq \frac{C_{\bar\de}^2}{(\Gamma_0+1)^2},$ then 
$E_t \leq \frac{M C_{\bar \de}}{\sqrt{N_t}}.$ Since $\Prob \{\beta \in \mathcal E_t(\bar \de)\}\geq 1-\bar\de$, we obtain 
$\Prob \left\{E_t \leq \frac{M C_{\bar \de}}{\sqrt{N_t}}\right\}\geq 1-\bar\de.$
\end{proposition}
We provide the proof in  \Cref{proof:dfs}.
From the result above it directly follows that $E_t = O\left(\frac{1}{t \ln t}\right)$, hence $\sum_{k=0}^tE_k = O\left(\ln\ln t\right)$. Using this result, and the classical FW proof technique (\cite{freund2016new}) we can conclude the result of \Cref{thm:conv}. This concludes the proof sketch.
\end{proofs}

We can see that under our choice of the number of measurements $n_t$ at each iteration, the total number of measurements  
is $O\left({d^2}t^2\ln t^2\right)$. It follows that the required number of measurements  {at} each step grows almost linearly with  the iteration number and quadratically with the dimension $d$. Note however that the number of iterations is independent of the dimension $d$. In contrast, the safe learning approach in \citep{pmlr-v37-sui15} is based on gridding the decision space and hence, the dependence in $d$ is exponential.  Hence, compared to previous safe learning approaches \citep{pmlr-v37-sui15, berkenkamp2016bayesian}, our method scales better with dimension. Naturally, this scalability is due to the assumption of convexity of the cost function and the linearity of the constraints.

{Finally, let us clarify some computational complexity aspects.
After adding each new data point to $X$, the matrix inversion $(X^TX)^{-1}$ in Step 6 can be performed using one-rank updates (e.g., using formula $(A+vv^T)^{-1} = A^{-1} - \frac{1}{(1+v^TA^{-1}v)}(A^{-1}vv^TA^{-1}))$. The cost of each such operation is $O(d^2)$. This operation is to be made $N_t = O(d^2 t^2 (\ln t)^2)$ times. The total computational complexity is thus $O(d^2N_t) = O(d^4t^2 (\ln t)^2)$ from Step 6 and additionally $t$ LP oracle calls in Step 7.}

\subsection{Extension to stochastic oracle for the objective function.}

Notice that the SFW algorithm requires $t = \tilde O\left(\frac{1}{\epsilon}
\right)$ iterations and $N_t = \tilde O\left(\frac{d^2\ln \frac{1}{\de}}{\epsilon^2}
\right)$ measurements of constraints to obtain a required accuracy of $\epsilon$. General Stochastic Frank-Wolfe algorithm with stochastic objective but known linear constraints require $t = O(\frac{1}{\epsilon})$ iterations and in contrast, $t = O(\frac{1}{\epsilon^3})$ stochastic gradient measurements 
{(\cite{hazan2016variance}, Table 2).}\footnote{The projection-free scheme for stochastic optimization proposed in \citep{lan2016conditional} achieves $N_t = O\left(1/\epsilon^2\right)$ measurements in total, but their method significantly differs from the original FW and the one used in the current paper.} This difference in the number of measurements is due to the fact that in the absence of linearity of the objective function, the gradients of the objective function are changing in each iteration. Thus,  $O(t^2)$ measurements at each iteration are needed to guarantee correct variance reduction rate of the Frank-Wolfe method (see (\ref{freund})).
From the above observation, we can extend the SFW analysis to the case in which we have access to a stochastic first-order oracle of the objective function. In this case, a total of $O(t^3)$ calls to the objective function oracle, and $O(t^2)$ calls to the constraints oracle are sufficient to obtain the desired rate of decrease of $E_t$ in \Cref{prop:dfs} and hence, the convergence rate 
in \Cref{thm:conv}. Similarly, for the case of zeroth-order oracle $O(d^2t^3)$ calls are needed to estimate the gradient of the objective.  The noisy gradient  does not influence the safety of the proposed algorithm.
Thus, the safety results in \Cref{thm:safety} extend to the case with stochastic first-order or zeroth-order oracle of the objective. 
\section{EXPERIMENTS}\label{sec:simulations}
We evaluate the performance of the proposed approach experimentally. In the first experiment, we consider the convergence rate of the algorithm as a function of the dimension. In the second experiment, we compare the SFW algorithm with a robust optimization based approach, which first learns the uncertain constraints and then finds the optimum with respect to the estimated constraints.  
{We consider the  convex smooth optimization problem:{
\begin{align*}
 \min_{x\in D }  \frac{1}{2}  \left\|x-x'\right\|^2_2,
\end{align*}} where  {$D = \left\{x\in \R^d
    |
   -1 \leq x^i \leq 1,
   i = 1,\ldots, d\right\}$} and $x' = [2,0.5,\ldots,0.5]\in \R^d$ for varying dimension $d$. 
   Then, the solution $x_*$ is a point on the boundary of the true constraint set above.  We set the variance of the noise to $\sigma = 0.01$ and use a constant exploration radius $\omega_0 = 0.01$. Furthermore, we set the confidence parameter  $ \de = 0.1$ and the total number of iterations to $T = 15$. 

\paragraph{Empirical constraint violation and convergence rate.}
The first experiment evaluates the empirical convergence rate and constraint violation as a function of the dimension  $d$. First, we evaluate the convergence rate assuming we can obtain the required lower bound on $C_n$ as per \Cref{thm:conv}.  We run the algorithm for dimensions $d = 2, 10, 20$. In particular, the parameters of the problem required for obtaining the lower bound are derived based on knowledge of the constraints and the objective function, as well as the input parameters $\de$, $T$ as follows: $\e_0 = 1,\bar \de = 0.0067, \phi_{\bar \de} = 3.43$, $\|b\|=2, \|a\|=1.$ It follows that a value of $C_n = d^2\cdot 24 $ achieves the required number of measurements. The SFW proposed in Algorithm 1 is  then run with the above choices of parameters $\bar \de$ and $n_t$. 
{For} each dimension, we run the SFW algorithm 20 times, keeping all the parameters and the initial conditions the same. The difference in each experiment is due to the stochastic noise in the measurements.  The average and standard deviation of the function values $f(x_t) - f(x_*)$ scaled by the initial condition error are shown in \Cref{plots6}. It can be seen that the dimension does not influence the convergence rate, rather, it influences only the number of measurements.

We also run the experiment assuming we cannot compute the lower bound $C_n$ precisely due to lack of problem data. In this case, at each iteration we first take $2dt$ measurements and further continuously take new measurements around $x_t$ until the safety set $S_{t+1}({\bar\de})$ grows sufficiently to ensure $x_{t+1}$ becomes safe (see 
Fact \ref{fact:fact1}). This safety indeed 
verifies the feasibility of iterates with high probability based on Fact \ref{lemm:safety}. Let us refer to this as the adaptive variant of SFW. 
This adaptive 
approach is not only more practical due to lack of dependence on problem data, but also requires far fewer measurements in total since the bound on $n_t$ from the \Cref{thm:safety} is quite conservative. This is due to the fact that our theoretical bounds were derived using the worst-case estimate of $\|\Sigma_t^{1/2}\|$, 
when the measurements are taken always around the same point. 
 However, $\|\Sigma_t^{1/2}\|$ can reduce much faster in practice. The convergence will also hold since the bound on $E_t \leq \frac{C_{\bar\de}M}{\sqrt{N_t}}\leq \frac{C_{\bar\de}M}{t+2} = O\left(\frac{1}{t}\right)$ required for the convergence rate is still satisfied. 
 {In} \Cref{plots6} caption we reported the required total number of measurements up to step $T$ of the adaptive variant by $N_T^a$. You can see that it has significantly reduced compared to the non-adaptive variant.

\begin{figure*}[ht]
\centering
\begin{tabular}{ccc}
\includegraphics[width = 0.28\textwidth]{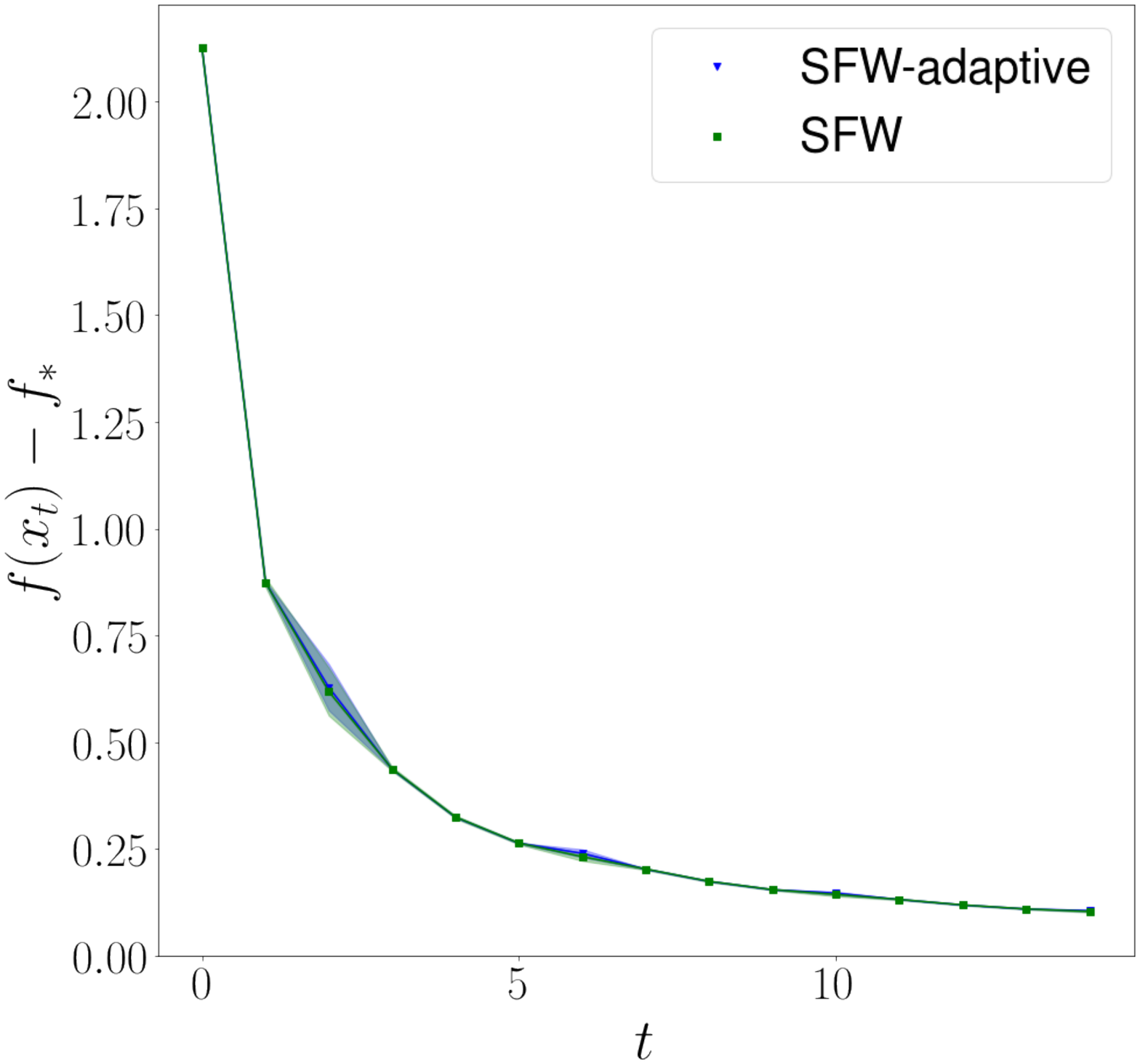}
&\includegraphics[width =0.28\textwidth]{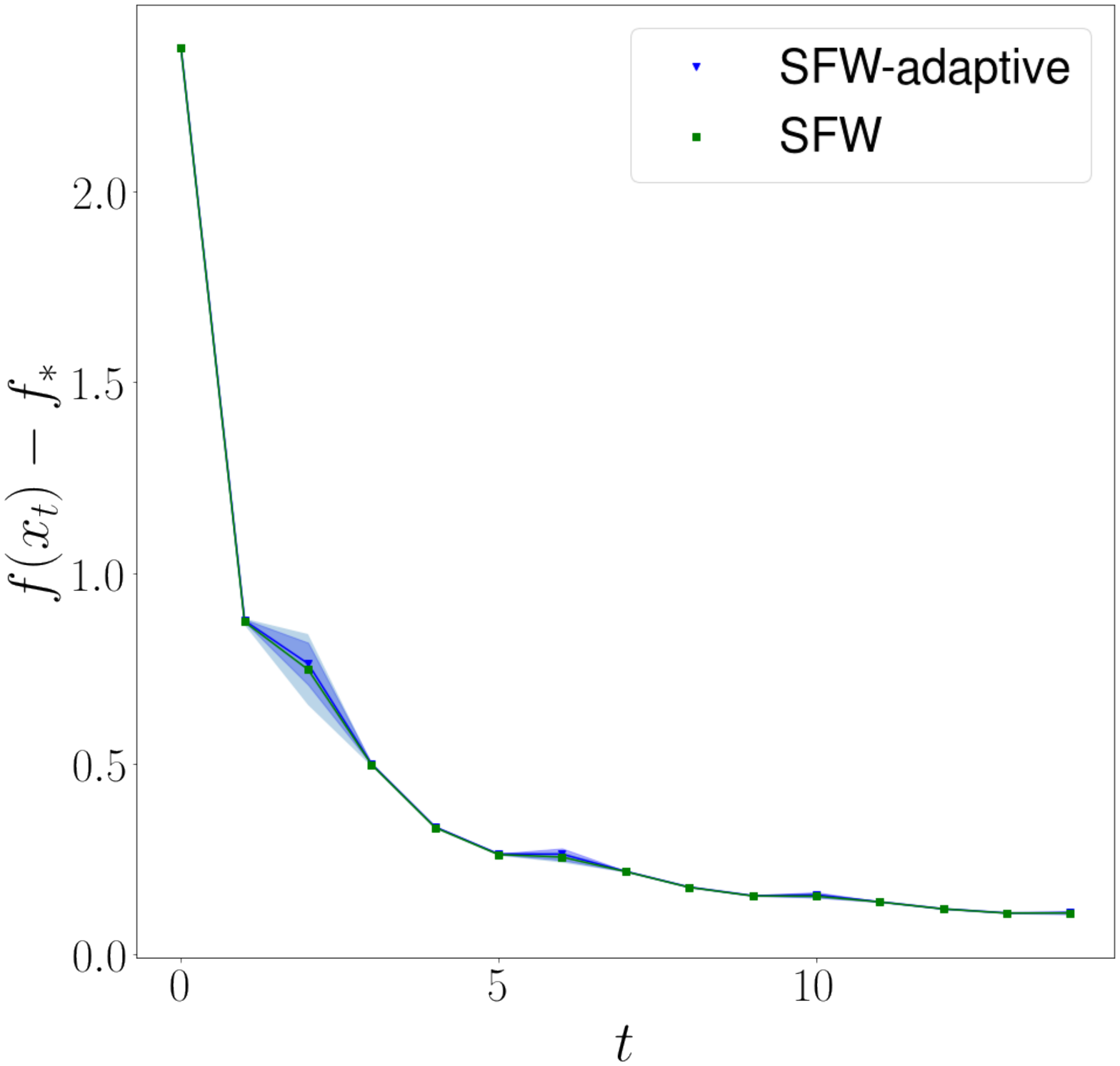}
&\includegraphics[width =0.28\textwidth]{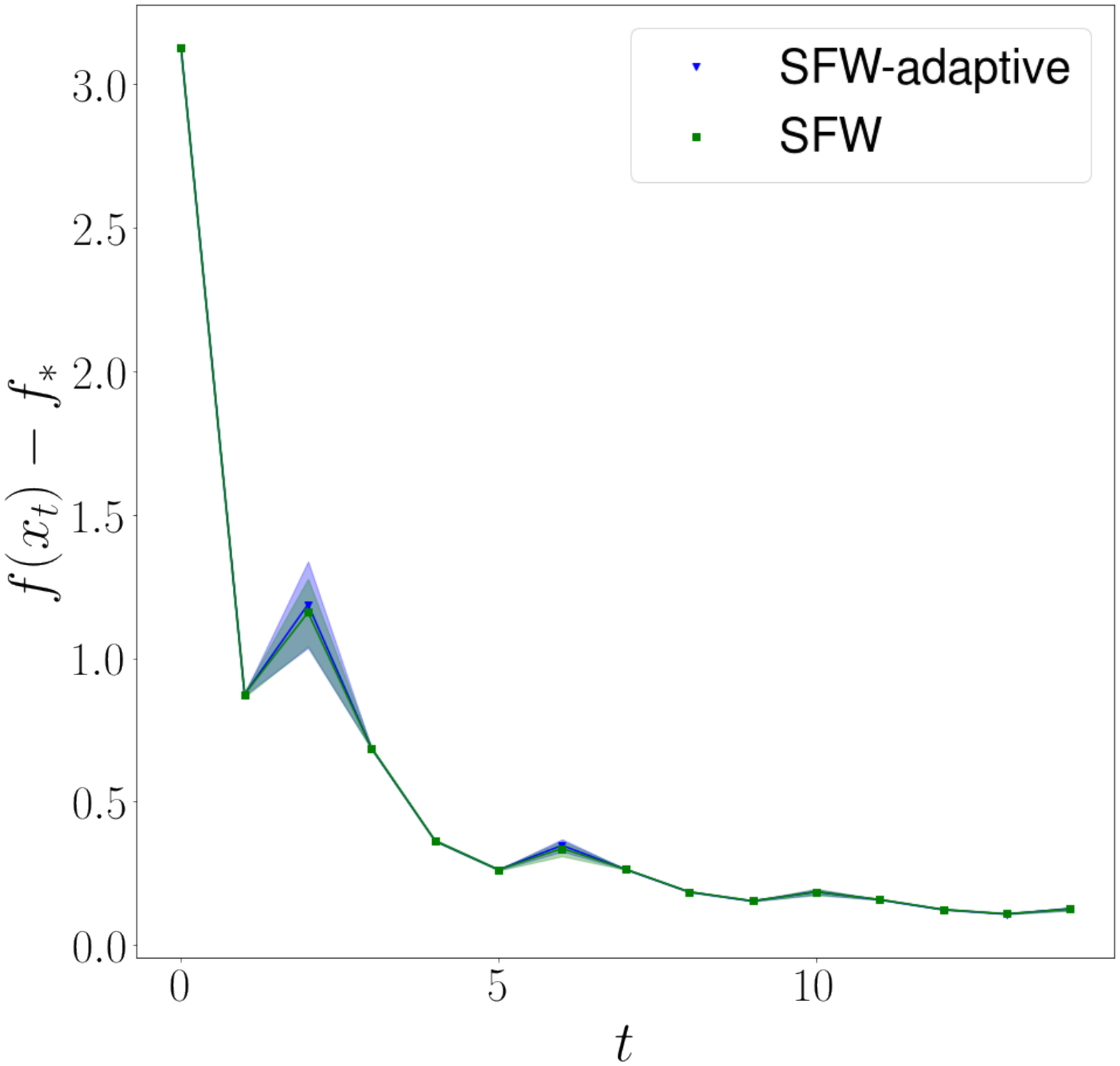}\\
(a) $d = 2,N_{T} = 8396, N_{T}^a = 519 $ & (b) $d = 4, N_T = 16792, N_T^a = 1135$ & (c) $d = 10, N_T = 41980, N_T^a = 4275$
\end{tabular}
\captionof{figure}{Convergence rate of SFW method for the dimensions $d=2,4,10$ with  $T = 15$.}
\label{plots6}
\end{figure*}


\paragraph{Comparison with an alternative robust optimization approach.}

\begin{figure*}[ht]
\centering
 \begin{tabular}{ccc}
\includegraphics[width = 0.28\textwidth]{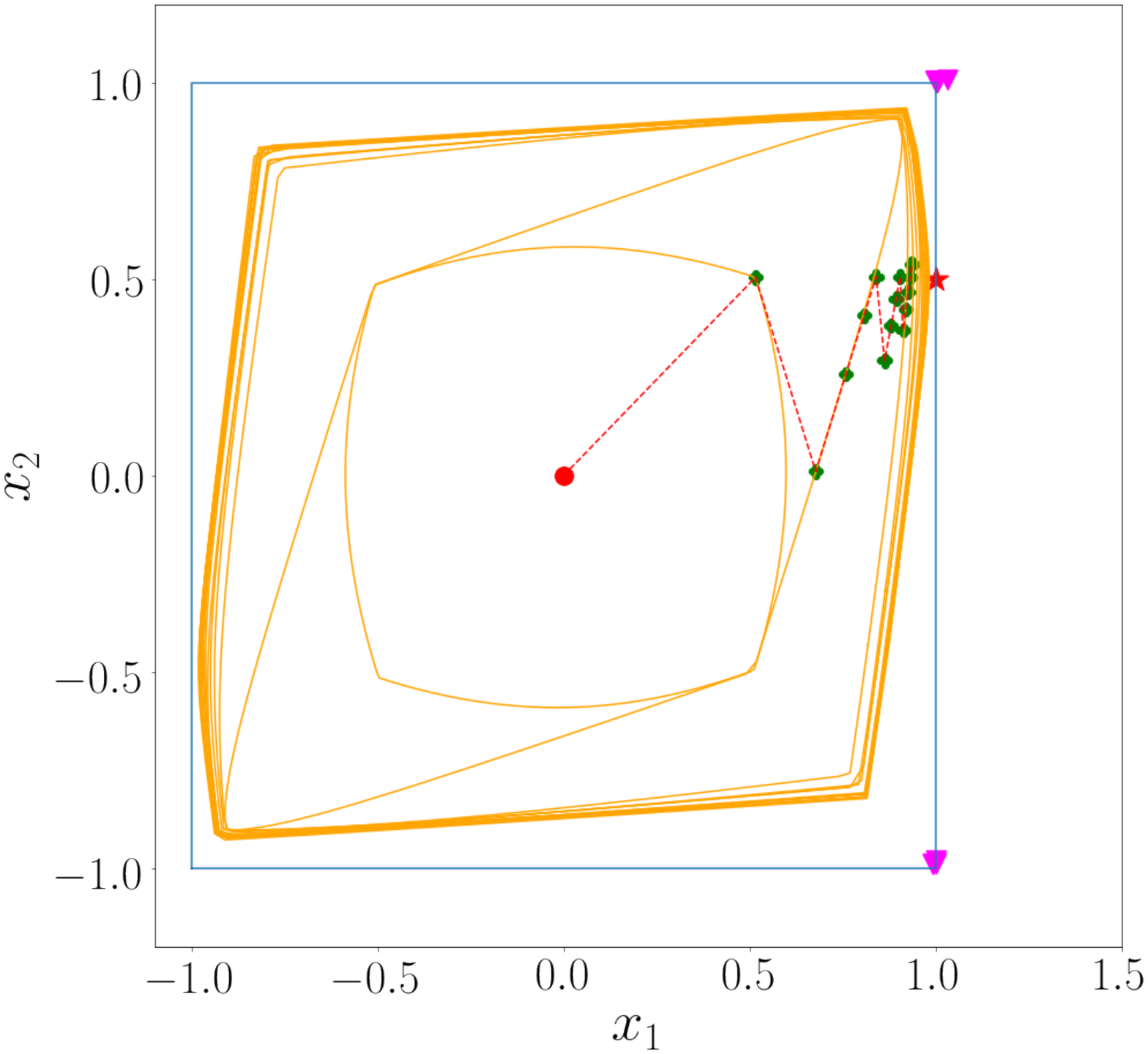}&
\includegraphics[width =0.28\textwidth]{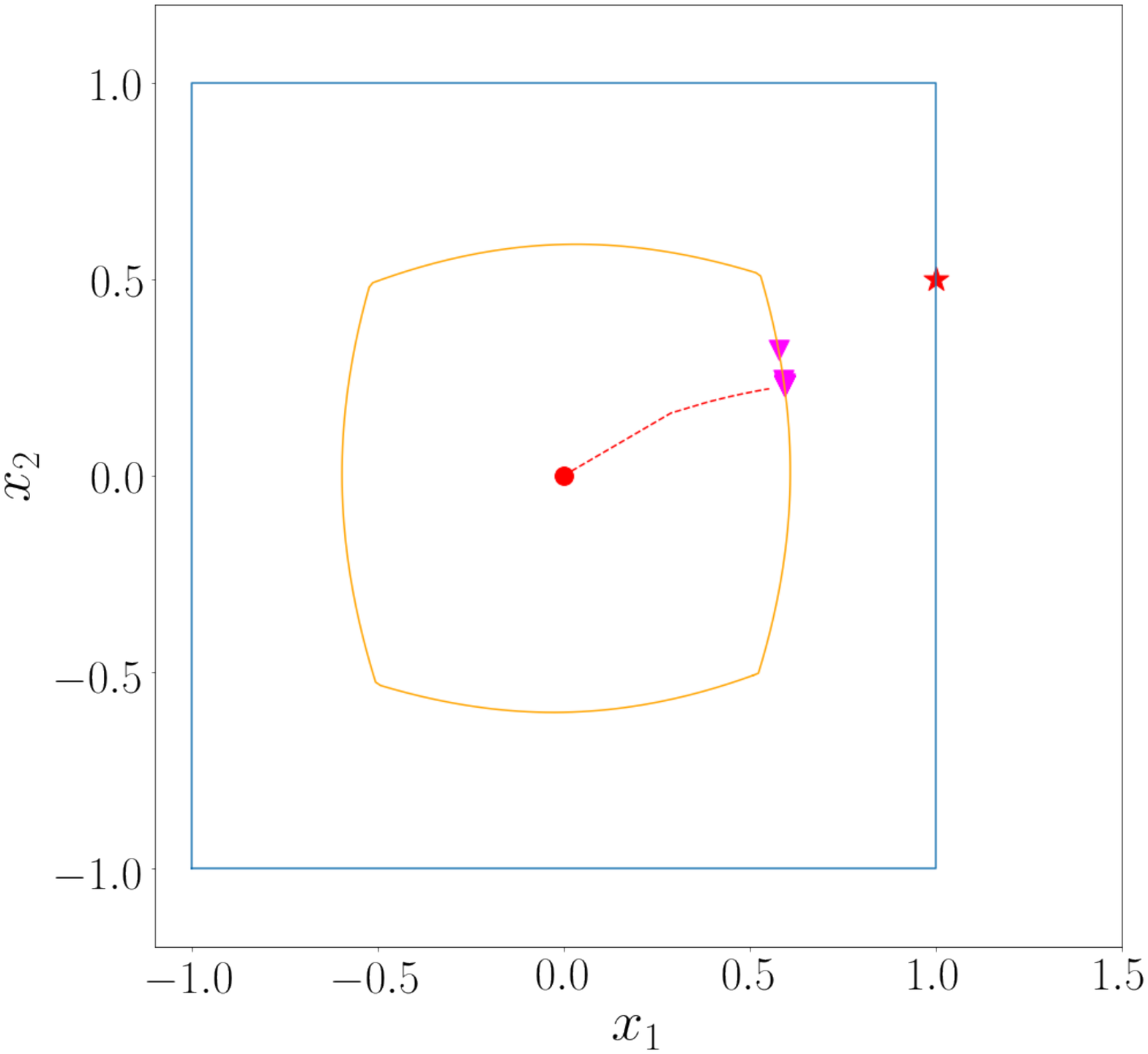}&
\includegraphics[width =0.27\textwidth]{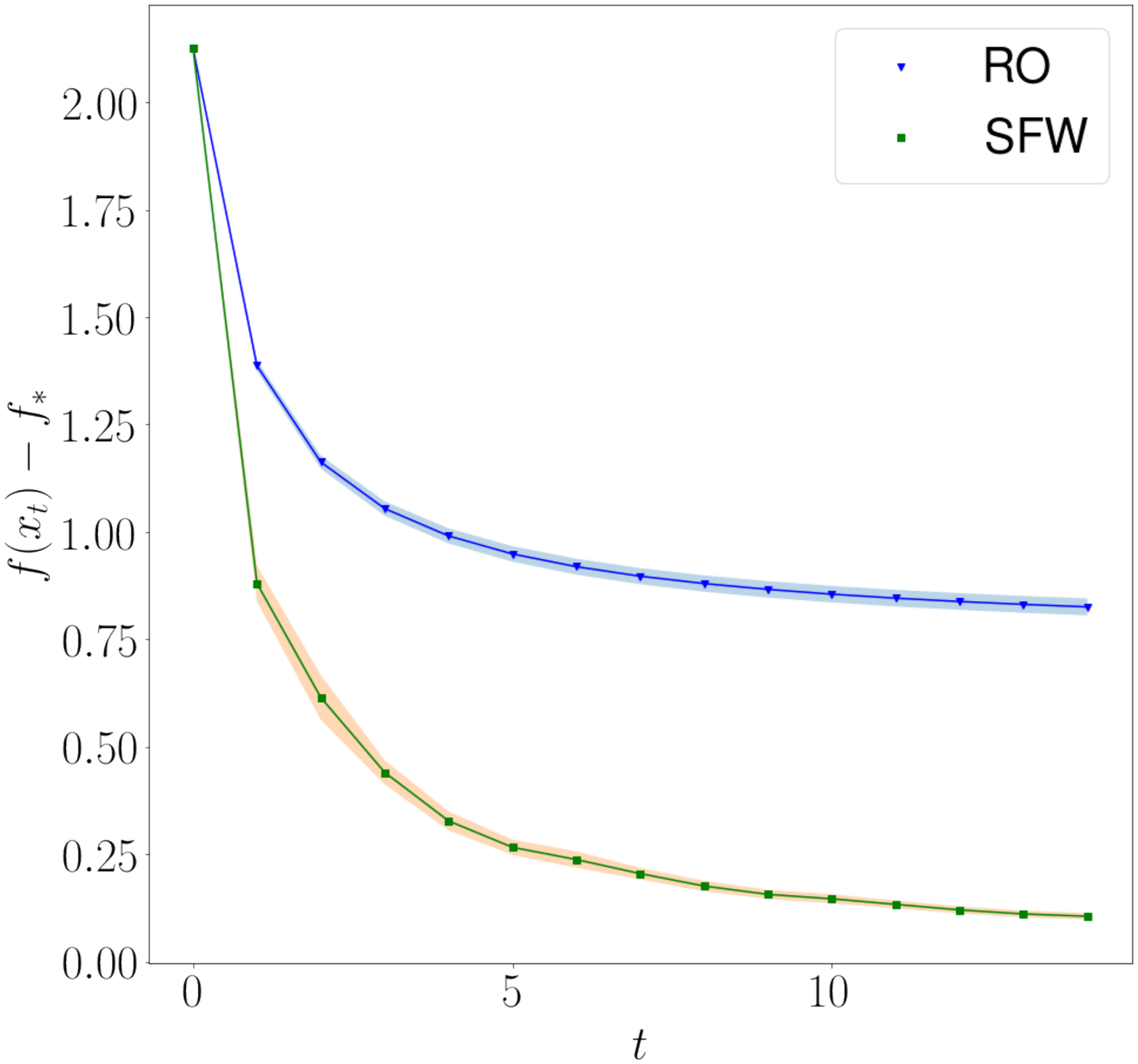}\\
(a) SFW& (b) RO& (c) Convergence rates of \\&&SFW and RO
\end{tabular}
\captionof{figure}{Trajectories and accuracy of the objective value by each iteration of optimization. The left pair of plots shows the convergence of one realization of SFW method and of the robust optimization approach (without safety set updates)
with $\sigma=0.1$,  $T = 15$, $N_{T} = 5500$. The orange lines denote the boundaries of the safety sets $S_t(\bar \de)$. Red circle denotes the starting point and star denotes the solution of the original problem. Magenta triangles denote the estimated DFS solutions.
}
\label{plots5}
\end{figure*}

We compare the proposed SFW method with an alternative approach in which we first make enough measurements in the a priori safe region to estimate the safety set (see Definition (\ref{eq:safety_set})) with sufficiently high probability. Next, we run a first-order method, such as FW with respect to the nonlinear set $S({\bar\de})$. Let us call this approach RO for robust optimization.  To compare these two methods we set an a priori number of measurements for the alternative RO method 
equal to the total number of measurements $N_t$, of SFW algorithm corresponding to $\de = 0.1$ and $T = 15$. After estimation of the safety set, we make $T=15$ iterations of the FW method with the constraint set ${S_T}({\bar\de})$. Thus,  the total number of measurement and the total number of optimization steps of the two methods are equal. During the run of the SFW, as per discussion in the above example, we reduce the number of measurements required to ensure safety at each iteration $t$, {online}. {Figure \ref{plots5}(a),(b) shows the optimization trajectories of each method. The 
{g}reen round points along the trajectory correspond to the points where the constraints were measured.  We also show the comparison of their convergence rates in Figure \ref{plots5}(c). As we can see, SFW algorithm performs better both in terms of estimates of the constraints and convergence rate. This difference in performance can be explained based on two observations.  First,  SFW moves measurements along the trajectory $\{x_t\}$, and this can lead to smaller variance of the estimates $\Sigma_t = \sigma^2\left(\bar X_t^T\bar X_t\right)^{-1}$. Hence, the measurements are more informative and the safety set $S_T({\bar\de})$  is larger. Second,  SFW algorithm is proven to converge to an  $\epsilon$-optimal solution corresponding to the true constraints. The RO approach however, can at the very best converge to an optimum with respect to a safety set estimated in advance. 
From the computational perspective, at each iteration the proposed SFW method requires an LP oracle, whereas  the alternative RO approach requires solving a second-order cone program. Hence, the SFW is more tractable. 

\section{CONCLUSION}\label{sec:conclusion}
We proposed a safe learning approach for convex costs and uncertain linear constraints.  This method uses information along the optimization trajectory to decrease the objective value and to explore an unknown feasible set. Meanwhile, it ensures feasibility for each iteration with high probability. We provided an analysis of convergence rate of {our }
algorithm, as well as of feasibility guarantees for its iterations. Our next steps are to generalize the results to nonlinear constraints and to provide performance guarantees in terms of regret. 


\section{ACKNOWLEDGEMENTS}
We thank Johannes Kirschner and Michel Baes for their helpful comments.
\addcontentsline{toc}{section}{Bibliography}
\bibliography{bibliography}{}
\bibliographystyle{apalike}
\clearpage
\appendix
\section{Proof of Fact~\ref{fact:fact1}}\label{proof:fact2}
\begin{proof}
Recall that the safety set $S_t(\bar \de)$ after iteration $t$ is defined by the following inequalities:
\begin{align}\label{INEQ1}
 S_t(\bar \de)& = \Bigg\{ x\in \R^d
 : \forall i=1,\ldots,m~\left[[\hat{a}^i_t]^T x - \hat{b}_t^i\right]+\nonumber\\&
+\phi^{-1}(\bar \de/m) \sigma\left\|
(\bar{X}_t^T\bar{X}_t)^{-1/2}\begin{bmatrix}x\\-1\end{bmatrix}\right\| \leq 0 \Bigg\}.
\end{align}
 
 Remember that $\bar{x}_t = \frac{X^T\mathbf{1}}{N}$ is an average of the measured points. Using the inversion formula for a block matrix, we obtain
\begin{align*}(\bar{X}_t^T\bar{X}_t)^{-1} &= \left[\begin{array}{cc}
      X_t^TX_t  & -X_t^T\mathbf{1} \\
      -\mathbf{1}^TX_t & N_t
\end{array}\right]^{-1} \\
&= \left[\begin{array}{cc}
      R_t  & R_t\bar{x}_t\\
      \bar{x}_t^TR_t 
      &\frac{1}{N_t}+\bar{x}_t^TR_t\bar{x}_t 
\end{array} \right],\end{align*}
where 
\begin{align}\label{R}
R_t &= \left[X_t^TX_t - N_t\bar{x}_t\bar{x}_t^T\right]^{-1}\nonumber \\
&= \left[\sum_{j=1}^{N_t}(x_{(j)} - \bar{x}_t)(x_{(j)} - \bar{x}_t)^T\right]^{-1}.
\end{align}

Let us denote by $\phi_{\bar\de} = \sigma\phi^{-1}(\bar\de/m)$
and by $\e^i_{t} = \hat{b}_t^i-(\hat a_t^i)^Tx_t$.
 Then, the $i$-th inequality in (\ref{INEQ1})  
can be rewritten as follows: \begin{align*}
\sqrt{\frac{\phi_{ \de}^2}{N_t}+ 
\phi_{\de}^2 (x - \bar{x}_t)^TR_t(x - \bar{x}_t)} \leq \e_t^i.
\end{align*}
Substituting $x=x_t$ to the above and combining the inequalities together, we obtain that the condition $x_t\in S_{t}(\bar\de)$ is equivalent to  \begin{align*}
 \phi_{\bar \de}\sqrt{\frac{1}{N_t}+ 
 (x_t - \bar{x}_t)^TR_t(x_t - \bar{x}_t)} \leq \min_{i=1,\ldots,m}\e_t^i.
\end{align*}
\end{proof}

\section{DFS solution proof}\label{proof:dfs}
{Let us recall that for the polytope $D\in \R^d$, by an active set $B$ we denote a set 
of indices of 
$d$ linearly independent constraints active in a vertex $V^B\in\R^d$ of $D$, i.e., $V^B = [A^B]^{-1}b^B$. Here, $A^B$ is a corresponding sub-matrix of $A$ and $b^B$ is the corresponding right-hand-side of the constraint.
The vertex estimate $\check V^B_t$ of a polytope is described by the system of linear equations $\hat{A}^B_tx = \hat{b}^B_t$.}
\begin{lemma}\label{lemm:1}
If $\beta \in \mathcal E_t(\bar \de)$ and $N_t  \geq \frac{C_{\bar\de}^2}{(\Gamma_0+1)^2}$, where $C_{\bar\de}$ is defined in (\ref{def:c}), then for any vertex $V^B$ defined by the active set $B$ and its estimate $\check V^B_t$ we have that the estimation error is bounded by
$$\|\check V^B_t - V^B\| \leq \frac{C_{\bar \de}}{\sqrt{N_t}},$$
where $C_{\bar\de} = \frac{2\phi_{\bar \de} d (\Gamma_0+1) }{\rho_{\min}[A^B]}\sqrt{\frac{\Gamma_0^2+1}{\omega_0^2} + 1}.$
\end{lemma}

\begin{proof}
 Since the LSE (Least Squares Estimation) is unbiased,
$$\E \hat{A}^B_t = A^B, ~\E \hat{b}^B_t = b^B.$$

Let us denote by $\zeta_{t} = \hat{b}_t^B - b^B$ the uncertainty in estimation of $b^B$, and $G_t = \hat{A}_t^B - A^B$ the uncertainty in estimation of $A^B$. 

Our aim is to bound the error of the vertex estimation $\|\check V^B_t - V^B\|$. Recall that 
\begin{align*} &\check V^B_t - V^B = [\hat{A}_t^B]^{-1}\hat{b}_t^B - [A^B]^{-1}b^B =\\
&= [A^B + G_t]^{-1}(b^B+\zeta_t) - [A^B]^{-1}b^B.
\end{align*}
Note that for any matrices $A,B$ it holds that
$$(A+B)^{-1} = A^{-1} - (I + A^{-1}B)^{-1}A^{-1}BA^{-1}.$$

Therefore, we can modify the expression for the $\check V^B_t - V^B$ as follows
\begin{align*}
&\check V^B_t - V^B = \\
 &= \left[[A^B]^{-1} - (I + [A^B]^{-1}G_t)^{-1}[A^B]^{-1}G_t [A^B]^{-1}\right](b^B+\zeta_t)\\& - [A^B]^{-1}b^B= \\
 &= [A^B]^{-1}b^B + [A^B]^{-1}\zeta_t  -\\&- (I + [A^B]^{-1}G_t)^{-1}[A^B]^{-1}G_t [A^B]^{-1}(b^B+\zeta_t) -\\& 
 - [A^B]^{-1}b^B =\\
& =  [A^B]^{-1}\zeta_t  - (I + [A^B]^{-1}G_t)^{-1}[A^B]^{-1}G_t [A^B]^{-1}(b^B+\zeta_t).
\end{align*}

The norm of the difference between the vertex $V^B$ of the set $D$ and its estimation can be bounded by
 \begin{align}\label{eq:bound}
 &\|\hat{V}^B_t - V^B\| \leq \underbrace{ \|[A^B]^{-1}\zeta_t\|}_{(a)}+\nonumber \\&+\underbrace{\| [A^B]^{-1}G_t\|}_{(b)}\underbrace{\|V^B + [A^B]^{-1}\zeta_t\|}_{(c)}\underbrace{\|\left(I + [A^B]^{-1}G_t\right)^{-1}\|}_{(d)}.
 \end{align}
 To obtain the bounds on the terms (a),(b),(c),(d), let us first obtain the bounds on $\|G_t\|$ and $\|\zeta_t\|$.
 
Assume that for each $i=1,\ldots,m$ $\beta^i \in \mathcal E^i_t(\bar \de)$, where $$\mathcal E^i_t(\bar \de) = \left\{z \in R^{d+1}:(\hat{\beta}_t^i - z)^T\Sigma_t^{-1}(\hat{\beta}_t^i - z)\leq \phi^{-1}(\bar \de)^2\right\},$$ i.e., that for any active set $B$ describing the vertex $V^B$ we have $\beta^B \in \mathcal E^B_t(\bar \de)$. 
{Consequently, $\|\hat a^i_t - a^i\|^2+|\hat b^i - b^i|^2 \leq \phi^{-1}(\bar \de)\|\Sigma_t^{1/2}\|~ \forall i  \in B.$} Then, for each row of $G_t$ we have $\|\hat a^i_t - a^i\|\leq \phi^{-1}(\bar \de)\|\Sigma_t^{1/2}\|$, and for each element of $\zeta_t$  we have $|\hat b^i_t - b^i|\leq \phi^{-1}(\bar \de)\|\Sigma_t^{1/2}\|.$ 
Hence, for $\|G_t\|$ we obtain
\begin{align}\label{G}
\|G_t\| \leq \|G\|_F = \sqrt{\sum_{i\in B}\|\hat a^i_t - a^i\|^2_2} \leq \sqrt{d}\phi^{-1}(\bar \de)\|\Sigma_t^{1/2}\|.
\end{align} 
Similarly, we obtain a bound on $\|\zeta_t\|$:
\begin{align}\label{zeta}\|\zeta_t\|  = \sqrt{\sum_{i\in B}(\hat b^i_t - b^i)^2} \leq \sqrt{d}\phi^{-1}(\bar \de)\|\Sigma_t^{1/2}\|.\end{align} 

For the LSE covariance matrix norm $\|\Sigma_t^{1/2}\|$ we have 
\begin{align*}
\|\Sigma_t^{1/2}\| &= \sigma\|(\bar X_t^T\bar X_t)^{-1}\|^{1/2} = 
\\&=\sigma \left\|\begin{bmatrix}I\\\bar x_t^T\end{bmatrix}R_t \begin{bmatrix}I&\bar x_t\end{bmatrix} + \begin{bmatrix}0&0\\0&1/N_t\end{bmatrix}\right\|^{1/2} \leq
\\&\leq \sigma\sqrt{\|R_t\|\|\bar x_t\|^2+ \|R_t\| + \frac{1}{N_t}},
\end{align*}
where $R_t$ was defined in (\ref{R}).  

Note that we make measurements as it is described in Step 4 of the SFW algorithm, i.e., we make measurements at all coordinate directions within small step size $\omega_0$ from points generated by the method.
Then, each new $2d$ measurements result in addition {of} a matrix  
$\sum_{j = N_t+1}^{N_t + 2d} (x_{(j)}-\bar x)(x_{(j)}-\bar x)^T \succeq \omega_0^2 I$ {to} the matrix $R_t^{-1}= \sum_{j=1}^{N_t} (x_{(j)}-\bar{x}_t)(x_{(j)}-\bar{x}_t)^T$. Hence, $R_t^{-1} \succeq \frac{N_t\omega_0^2}{2d} I$.
Hence, the minimal eigenvalue of the covariance matrix $R_t^{-1}$  is bounded from below by the value $\lambda_{\min}(R_t^{-1})\geq \frac{ N_t\omega_0^2}{d}.$ 
Thus, we obtain the following bound on the norm of $R_t$:  \begin{align}\label{R_t} \|R_t\|\leq \frac{d}{N_t\omega_0^2}.\end{align} 
Note that $\bar x_t \in D$, hence $\|\bar x_t\|\leq \Gamma_0$.
It follows that \begin{align}\label{sigma}
\|\Sigma_t^{1/2}\|&\leq \sigma\sqrt{\|R_t\|\|\bar x_t\|^2+ \|R_t\| + \frac{1}{N_t}}\leq\nonumber\\
&\leq\sigma\sqrt{\frac{d}{N_t\omega_0^2}\|\bar x_t\|^2+ \frac{d}{N_t\omega_0^2} + \frac{1}{N_t}}\leq\nonumber\\
&\leq\frac{\sigma \sqrt{d}
\sqrt{\frac{\Gamma_0^2+1}{\omega_0^2}+\frac{1}{d}
}}{\sqrt{N_t}} \leq \frac{\sigma \sqrt{d}
\sqrt{\frac{\Gamma_0^2+1}{\omega_0^2}+1
}}{\sqrt{N_t}} .    
\end{align}

In order to bound terms (a),(b),(c),(d) in inequality (\ref{eq:bound}), let us also bound the norm of the matrix $[A^B]^{-1}$:
\begin{align}\label{Ab}\|[A^B]^{-1}\|=\rho_{\max}([A^B]^{-1}) =  
\frac{1}{\rho_{\min}[A^B]}\leq \frac{1}{\rho_{\min}(D)}.\end{align}
Then, combining inequalities (\ref{G}),(\ref{zeta}),(\ref{sigma}),(\ref{Ab}), we bound terms (a) and (b) as follows:
$$(a): \left\|[A^B]^{-1}\zeta_t\right\| \leq\|[A^B]^{-1}\|\|\zeta_t\|\leq \frac{U}{\sqrt{N_t}},$$
$$(b): \left\|[A^B]^{-1}G_t\right\| \leq \|[A^B]^{-1}\|\|G_t\| \leq \frac{U}{\sqrt{N_t}},$$
where $U$ is defined by  $$U = \frac{\phi_{\bar \de} d }{\rho_{\min}(D)}\sqrt{\frac{\Gamma_0^2+1}{\omega_0^2} + 1}.$$
Further, let us bound term (d). 
For $N_t 
\geq
4U^2$
it holds that
\begin{align*}
&\left\|[A^B]^{-1}G_t\right\| \leq \frac{U}{\sqrt{N_t}} \leq \frac{1}{2},\\&\left\|[A^B]^{-1}\zeta_t\right\| \leq \frac{U}{\sqrt{N_t}}\leq  \frac{1}{2}.\end{align*}
As  such, for $N_t \geq 4U^2$ we have 
\begin{align*}\|
&\left(I + [A^B]^{-1}G_t \right)^{-1}\| = 
\|I^{-1} -  \left(I + [A^B]^{-1}G_t \right)^{-1} [A^B]^{-1}G_t\|\leq \\
&\leq \|I\| + \|[A^B]^{-1}G_t\|\|\left(I + [A^B]^{-1}G_t \right)^{-1}\|\leq \\
&\leq 1 + \frac{1}{2} \|\left(I + [A^B]^{-1}G_t \right)^{-1}\|
\end{align*}
Hence 
\begin{align}\label{(d)}\left\|\left(I + [A^B]^{-1}G_t \right)^{-1} \right\| \leq 2.\end{align}
Finally, term (c) can be bounded as follows:
\begin{align}\label{(c)}\left\|V^B + [A^B]^{-1}\zeta_t \right\| \leq \Gamma_0 + \frac{U}{\sqrt{N_t}}.
\end{align}
Combining these all together, we obtain
\begin{align*}
 &\|\hat{V}^B_t - V^B\| \leq  \frac{U}{\sqrt{N_t}}+  2\frac{U}{\sqrt{N_t}}\left(\Gamma_0 +  \frac{U}{\sqrt{N_t}}\right)\leq\\&\leq 2\frac{U}{\sqrt{N_t}}\left(\Gamma_0 + 1\right) = \frac{C_{\bar \de}}{\sqrt{N_t}},
 \end{align*}
 where 
 $$C_{\bar\de} = 2U(\Gamma_0+1) = \frac{2 \phi_{\bar \de} d (\Gamma_0+1) }{\rho_{\min}(D)}\sqrt{\frac{\Gamma_0^2+1}{\omega_0^2} + 1}.$$


Since $N_t = C_n t^2 (\ln t^2) \geq C_n$, the above bound holds
under the proper choice of the constant $$C_n \geq 4U^2 =  \frac{C_{\bar\de}^2}{(\Gamma_0+1)^2} = \frac{4 d^2 \phi^2_{\bar \de}}{\rho^2_{\min}(D)} \left(\frac{\Gamma_0^2+1}{\omega_0^2} + 1\right).$$

\end{proof}
\paragraph{\Cref{prop:dfs}}
\begin{proof}
Let us bound the difference between the solution of the estimated DFS and the solution of the true DFS. Estimated DFS is a linear program defined by \begin{align*}
\hat{s}_t = &\arg\min_{\hat{A}_tx \leq \hat{b}_t}
\la  c_t,x \ra,
\end{align*}
where $c_t = \nabla f(x_t)$. 
Any solution of such a linear program is a vertex {(or convex hull of vertices)} of the polytope $\hat{A}_tx \leq \hat{b}_t$. Let 
$\hat V^2_t$ be the estimated DFS solution vertex $\hat V^{2}_t  =\hat s_t  = \arg\min_{s\in\hat D_t}\la  c_t,s \ra.$ And correspondingly, let 
 $V^{1}_t$ be the true DFS solution $V^{1}_t  =  s_t = \arg\min_{s\in D}\la c_t,s \ra.$ 

Let us define by $\pi_{\hat D_t} V^1_t$ the projection of $ V^1_t$ onto $\hat D_t$: $\bar V^1_t = Proj_{\hat D_t} V^1_t,$ 
and correspondingly we define $\tilde V^2_t$ as $\tilde V^2_t = Proj_{D}\hat V^2_t$.  
 Recall that the  estimate $\check V^B$ of any vertex $V^B = [A^{B}]^{-1}b^B$ of the polytope $D$ is described by the system of linear equations $\hat{A}^B_tx = \hat{b}^B_t$. 
 Since the estimates $\check V^B$ denote intersections of the hyper-planes $\la \hat a^j, s\ra = \hat b^j~\forall j\in B$ for some particular subset of indices $B$, the polytope $\hat D_t$ lies inside the convex hull of the estimates $\check V^B_t.$ 
 Hence, any point $s \in \hat D_t$ cannot be further from $D$ than the estimates of all the vertices $\check V^i$ from $D$. By \Cref{lemm:1}, if $N_t  \geq \frac{C_{\bar\de}^2}{(\Gamma_0+1)^2}$ with $C_{\bar\de}$ defined in (\ref{def:c}) and $\beta\in \mathcal E_t(\bar \de)$, then for any vertex $V^B$ of $D$ we have $\|V^B  -  \check V^B_t\| \leq \frac{C_{\bar \de}}{\sqrt{N_t}}$, thus we obtain that the distance from any point $s \in \hat D_t$ to $D$ is less than $\frac{C_{\bar \de}}{\sqrt{N_t}}$. I.e., we have $\| \bar V^1_t-V^1_t \| \leq \frac{C_{\bar \de}}{\sqrt{N_t}}.$

Similarly, we show that the distance from any point $s\in D$ to the set $\hat D_t$ is upper bounded by $ \frac{C_{\bar \de}}{\sqrt{N_t}}$. Again, we can see that $D$ is bounded by the convex hull of $V^{\hat B}$, where each $V^{\hat B} = [A^{\hat B}]^{-1}b^{\hat B}$ corresponds to a vertex $\check V^{\hat B}_t = [\hat A_t^{\hat B}]^{-1}\hat b_t^{\hat B}$ of $\hat D_t$. Hence , the distance from any point $s\in D$ to the set $\hat D_t$ is upper bounded by $\max_{\hat B}\| V^{\hat B}-\check V^{\hat B}_t\|.$ Then, by \Cref{lemm:1} we have  $\|\tilde V^2_t-\hat V^2_t\| \leq \frac{C_{\bar \de}}{\sqrt{N_t}}$.

Note that  $\bar V^1_t\in \hat D_t$, $\tilde V^2_t\in D$.  From the definitions of $\hat V^2_t, V^1_t$ above it follows that
\begin{align*}&c_t^T\hat V^2_t\leq c_t^T\bar V^1_t,\\
&c_t^T V^1_t\leq c_t^T \tilde V^2_t.\end{align*} 
 Hence, we have \begin{align*}
     &c_t^T\hat V^2_t - \|c_t\|\frac{C_{\bar \de}}{\sqrt{N_t}} \leq c_t^T\tilde V^1_t - \|c_t\|\frac{C_{\bar \de}}{\sqrt{N_t}}\\&\leq c_t^T V^1_t \leq c_t^T \tilde V^2_t \leq c^T\hat V^2_t + \|c_t\|\frac{C_{\bar \de}}{\sqrt{N_t}}. 
 \end{align*}
 Thus, we obtain that if $\beta\in \mathcal E_t(\bar \de)$, then $$E_t = |c_t^T(\hat{s}_t-s_t)| = |c_t^T\hat V^2_t-c_t^TV^1_t| \leq \|c_t\|\frac{C_{\bar \de}}{\sqrt{N_t}}.$$

Note that $\|c_t\|\leq M$, where $M$ is the Lipschitz constant of the objective. Thus, if $\beta\in \mathcal E_t(\bar \de)$, then $E_t \leq \frac{C_{\bar \de} M}{\sqrt{N_t}}$, i.e.,
\begin{align*}
&\Prob \left\{E_t \leq \frac{C_{\bar \de} M}{\sqrt{N_t}}\right\} \geq 1-\bar\de.%
\end{align*}
\end{proof}
\section{Proof of \Cref{prop:gamma}}\label{proof:lemma1}
First, we provide some preliminary lemmas for the proof of \Cref{prop:gamma}.

Let us denote by
$ {\breve{x}_t} =  {x_{t}}-\bar x_t$,
$\Delta^k_t = \hat{s}_{ {t}} -   {x_{k}}$, and recall that $\e^i_{t} = \hat{b}_t^i - [\hat{a}_t^i]^T {x_{t}}$. 
By $\hat V_{t-1}$ we call the solution of the estimated DFS at the step $t-1$, by $\tilde V_{t-1} = Proj_{ D}\hat V_{t-1}$, and by $\bar V_{t,t-1} = Proj_{\hat D_t}V_{t-1}$. By $B_{t}$ we denote the active set corresponding to $\hat V_{t}$  (see Lemma 1 for definition) and by $B_{t-1}$ the active set corresponding to $\hat V_{t-1}$  
 
\begin{lemma}\label{lemm:3}
If $\beta\in \mathcal E_t(\bar \de)\cap \mathcal E_{t-1}(\bar \de)$ holds, then we have
\begin{align*}
    \min_i\la \hat{a}_t^i, \Delta_{t}^{t}\ra \geq&  (1-\gamma_{t-1}) \min_i \la \hat{a}_{t}^i, \Delta_{t}^{t-1}\ra -\\&- \frac{2\gamma_{t-1}\max_i\|\hat{a}_t^i\|C_{\bar\de}}
    {\sqrt{N_{t-1}}}.
\end{align*}
\end{lemma}
\begin{proof}
\begin{align*}
    \forall \hat{s}_t : \min_i\la \hat{a}_t^i,  \Delta^{t}_t\ra &= \min_i\la \hat{a}_t^i, \hat{s}_t - x_{t}\ra  = \min_{i\in B_t}\la \hat{a}_t^i, \hat{s}_t - x_{t}\ra =\\
    &= \min_{i \in B_t}\la \hat{a}_t^i, \hat{s}_t - x_{t-1} - \gamma_{t-1}(\hat s_{t-1} - x_{t-1})\ra \\
    &= \min_{i \in B_t}\la \hat{a}_t^i, (1-\gamma_{t-1})(\hat{s}_t - x_{t-1}) + \\
    &+\gamma_{t-1}(\hat s_{t-1} - \hat s_t)\ra=\\ (\text{Let } j = \arg\min_{i\in B_t}&\la \hat{a}_t^i,  \Delta_t^{t-1}\ra)\\
    &=(1-\gamma_{t-1}) \la \hat{a}_t^j, \hat{s}_t - x_{t-1}\ra  + \\&
    +\gamma_{t-1}\la \hat a_t^j,\hat s_{t-1} - \hat s_t\ra. \end{align*}
    For the point $s_t$ we have $\la \hat a^j_t, s_t\ra = 0$ and for any point $s\in \hat D_{t}$  we have 
    $\la \hat a_t^j,s \ra \geq 0$. Note that   $\hat V_{t,t-1} \in \hat D_t$ and that $\|\hat V_{t-1} - \hat V_{t,t-1}\|\leq \|\hat V_{t-1} -V_{t-1}\|+\| V_{t-1} - V_{t,t-1}\|.$

    \begin{align}\label{eq:2}&\min_i\la \hat{a}_t^i, \Delta_{t}^{t}\ra\geq(1-\gamma_{t-1}) \la \hat{a}_t^j, \hat{s}_t - x_{t-1}\ra  + \gamma_{t-1}\la \hat a_t^j,\hat V_{t-1}\ra\nonumber\\
    &\geq (1-\gamma_{t-1}) \la \hat{a}_t^j, \hat{s}_t - x_{t-1}\ra  + \gamma_{t-1}\la \hat a_t^j,\hat V_{t-1} - \hat V_{t,t-1}\ra\nonumber \\
    &\geq(1-\gamma_{t-1}) \min_i\la \hat{a}_t^i, \hat{s}_t - x_{t-1}\ra -\nonumber\\
    & \gamma_{t-1}\max_i\|\hat{a}_t^i\| \|\hat V_t - \hat V_{t,t-1}\| .
\end{align}
Using the result of \Cref{lemm:1} 
we can obtain that if $\beta\in \mathcal E_t(\bar \de)\cap \mathcal E_{t+1}(\bar \de)$ and $N_t  \geq \frac{C_{\bar\de}^2}{(\Gamma_0+1)^2}$ with $C_{\bar\de}$ defined in (\ref{def:c}), and applying the same arguments as in the proof of Proposition 1, we obtain
\begin{align}\label{eq:1}
   \|\hat{V}_{t-1} - \hat{V}_{t,t-1}\| &\leq \|\hat{V}_{t-1} - V_{t-1}\|+  \|V_{t-1} - \hat{V}_{t,t-1}\| \nonumber\\
   & \leq \frac{C_{\bar \de}}{\sqrt{N_t}}+\frac{C_{\bar \de}}{\sqrt{N_{t+1}}}.\end{align}
Then, combining (\ref{eq:1}) with (\ref{eq:2}), we have
\begin{align*}
  &\min_i\la \hat{a}_t^i,  \Delta^{t}_t\ra  \geq\\&
  \geq (1-\gamma_{t-1}) \min_i\la \hat{a}_t^i, \Delta_{t}^{t-1} \ra -\frac{2\gamma_{t-1}\max_i\|\hat{a}_t^i\|C_{\bar\de}}{ \sqrt{N_{t-1}}}.
\end{align*}
\end{proof}
{\Cref{lemm:3} above is an induction step in the proof of \Cref{lemm:4}.}
Lemma \ref{lemm:4} below bounds the fastest rate of decreasing the distance to the boundaries of $D$ for the SFW algorithm. Recall that $\mathcal{F}_t = \{\beta \in \cap_{k=0}^{t} \mathcal E_{k}(\bar \de)\}.$ 
\begin{lemma}\label{lemm:4}

If $\F_t$ holds, then we have
\begin{align*}
    \min_i\la \hat{a}_t^i, \Delta_{t}^{t}\ra 
    \geq \frac{\min_i\la \hat{a}_t^i , {{\Delta}_t^0} \ra}{t+2} \left(1 - \frac{C_{\bar \de}\ln \ln {t} \max_i\|\hat a^i_t\| }{\sqrt{C_n } {\min_i\la \hat a_t^i,\bar \Delta_0}\ra}\right).
\end{align*}
\end{lemma}

\begin{proof}
By induction, from Lemma \ref{lemm:3}  we have
 \begin{align*}
 &\min_i\la\hat{a}_t^i, \hat s_t - x_t\ra  \geq \prod_{j=0}^{t-1}(1-\gamma_j)\min_i\la \hat{a}_t^i ,x_{0} - \hat s_t \ra -\\
 &- \sum_{j = 0}^{t-1}\frac{2C_{\bar\de}\gamma_j}{\sqrt{N_j}}\max_i\| \hat{a}_t^i\|\prod_{k = j}^{t-1}(1-\gamma_k).
  \end{align*}
Note that $1-\gamma_k = \frac{k+1}{k+2}$, and
 \begin{align*}
 &\prod_{k=j}^{t-1}(1-\gamma_k) = \frac{(t)!/(j+1)!}{(t+1)!/(j+2)!} = \frac{j+2}{t+1}.
 \end{align*}
 \begin{align*}&\text{Thus, }\min_i\la \hat{a}_t^i, \hat s_t - x_t\ra\geq\\  &\geq \frac{1}{t+1} \min_i\la \hat{a}_t^i ,x_{0} - \hat s_t \ra - \sum_{j=0}^{t-1}\frac{j+2}{t+1}\frac{C_{\bar\de}\gamma_j}{\sqrt{N_j}}\max_i\| \hat{a}_t^i\|= \\
 &= \frac{1}{t+1}\min_i \la \hat{a}_t^i ,x_{0} - \hat s_t \ra - \frac{1}{t+1}\sum_{j=0}^{t-1}\frac{C_{\bar\de}}{\sqrt{N_j}}\max_i\| \hat{a}_t^i\|.
\end{align*}
Recall that $N_t = C_n t^2(\ln t)^2$, hence we have
\begin{align*}
&\e_t^i\geq\min_j\la \hat{a}_t^j, \hat s_t - x_t\ra  \geq  \frac{1}{t+2} \min_j\la \hat{a}_t^j , \hat s_t - x_0\ra -\\
&- \frac{1}{t+2}\sum_{j = 0}^{t}\frac{\sqrt{2\ln (j+1)+\ln 1/\bar\de}C_{\bar\de}}{\sqrt{C_n } (j+1) \ln (j+1)}\max_j\| \hat{a}_t^j\| =\\
&= \frac{1}{t+2}\left(\min_j\la \hat{a}_t^j , {{\Delta}_t^0} \ra - \frac{C_{\bar \de} \ln (\ln t) }{\sqrt{C_n }}\max_j\|\hat{a}_t^j\|\right),\\
& \textit{where $ {{\Delta}_t^0} = \hat{s}_t - x_0$. }
\end{align*}
\end{proof}
With Lemmas 3 and 4 in place, we are ready to prove \Cref{prop:gamma}.
 \subsection{Proof of \Cref{prop:gamma}}\label{proof:lemma1:1}

\begin{proof}
From Fact 2, the condition $x_t \in S_t(\bar \de)$ is equal to
$$\frac{\phi_{\bar \de}^2}{N_t} + \phi_{\bar \de}^2(x_t- \bar x_{t})^TR_{t}(x_t-\bar x_{t}) \leq \min_i[\e^i_{t}]^2.$$
From the bound on $\|R_t\|$ given in (\ref{R_t}) and knowing that $\Gamma$ is a diameter of the set $D$, we have
 $$\frac{\phi_{\bar \de}^2}{N_t} + \phi_{\bar \de}^2(x_t- \bar x_{t})^TR_{t}(x_t- \bar x_{t}) \leq \frac{\phi_{\bar \de}^2\left(1+\frac{d\Gamma^2}{\omega_0^2}\right)}{N_t}.$$
 From \Cref{lemm:4} and recalling that $\e_t^i = \min_i\la \hat a^i_t,\Delta_t^t\ra$, we have \begin{align}\label{eq:e}[\e_t^i]^2 \geq \frac{1}{(t+2)^2}\left(\min_i\la \hat{a}_t^i , {{\Delta}_t^0} \ra - \frac{C_{\bar \de} \ln (\ln t) }{\sqrt{C_n }}\max_i\|\hat{a}_t^i\|\right)^2.\end{align} 

Hence, we can guarantee that $x_t \in S_t(\bar \de)$ if 
\begin{align}\label{nt_bound}N_t \geq \frac{(t+2)^2\phi_{\bar \de}^2\left(1+\frac{d\Gamma^2}{\omega_0^2}\right)}{\left(\min_i\la \hat{a}_t^i , {{\Delta}_t^0} \ra - \frac{C_{\bar \de} \ln (\ln t) }{\sqrt{C_n }}\max_i\|\hat{a}_t^i\|\right)^2}.\end{align}

We denote by $L_{A} = \max_i\|a_i\|$.
Let us derive how far are $\min_i\la \hat{a}_t^i , {{\Delta}_t^0} \ra $ from $\e_0$ and $ \max_i\|\hat{a}_t^i\|$ from $L_{A}$. These are needed for obtaining a bound on the denominator above.
Let us define by $\Delta_0 = s_t - x_0$, where $s_t$ is the true vertex of $\hat s_t$ corresponding to  $\min_i\la \hat{a}_t^i , {{\Delta}_t^0} \ra$. Also recall that then $\min_i[\varepsilon_0] \leq \la a^i,\Delta_0\ra$. 
If $C_n \geq  \frac{C_{\bar\de}^2}{(\Gamma_0+1)^2}$, then with probability greater than $1 -\bar \de$ we have 
$\|{\Delta}_t^0  - \Delta_0\|\leq \frac{C_{\bar \de}}{\sqrt{N_t}}$(using  \Cref{lemm:1}). We also can bound the difference $\|\hat a^i_t - a^i\|$ by $\|\hat a^i_t - a^i\| \leq \phi^{-1}(\bar \de)\|\Sigma^{1/2}\| \leq \frac{C_{\bar\de}}{\sqrt{N_t}}\frac{1}{\sqrt{d}\rho_{\min}(D)(\Gamma_0+1)}.$ The second inequality follows from (\ref{sigma}) and definition of $C_{\bar\de}$ (\ref{def:c}).

 Combining above inequalities together with the bound (\ref{nt_bound}) on $N_t$ we can conclude the following. If
$$C_n\geq \frac{4 C_{\bar \de}^2 (\ln \ln T)^2 
L_{A}^2}{ \min_i[\e^i_0]^2}, $$
then we can guarantee that $x_t \in S_t(\bar \de)$ by {requiring}
$$N_t \geq \frac{(t+2)^2\phi_{\bar \de}^2\left(1+\frac{d\Gamma^2}{\omega_0^2}\right)}{\min_i[\e^i_0]^2}.$$
Since $n_t = C_n(t+1)(\ln (t+2))^2$ and $N_t = \sum_{k=0}^t n_k$, we obtain that $$N_t \geq C_n (t+1)^2(\ln (t+2))^2.$$  Hence, $C_n \geq \frac{\phi_{\bar \de}^2\left(1+\frac{d\Gamma_0^2}{\omega_0^2}\right)}{(\ln (t+2))^2\min_i[\e^i_0]^2}$ 
 is enough to ensure that $x_t\in S_t(\bar \de)$.
Note that  $C_{\bar \de} \geq \phi_{\bar \de}^2\left(1+\frac{d\Gamma^2}{\omega_0^2}\right)$.
Thus, under the proper choice of constant parameter $C_n$, namely, \begin{align*}C_n \geq 
\max\left\{\frac{ 
4 C_{\bar \de}^2 (\ln \ln T)^2L_{A}^2 }{[\e_0]^2},\frac{C_{\bar\de}^2}{(\Gamma_0+1)^2}\right\}\end{align*} we conclude that 
$x_t \in S_t(\bar \de).$
\end{proof}
\paragraph{Remark}
Note that if we use a step size as in classical FW, $\gamma_t = \frac{2}{t+2}$,  or in more general form  $\gamma_t = \frac{l}{t+l}$ then we obtain that the distance to the boundaries $\min_i[\e_t^i]$ will decrease with the rate upper bounded by  $\prod_{k=0}^t(1-\gamma_k) = \frac{l!}{t\cdot \ldots \cdot (t+l)} = O(\frac{1}{t^l})$ instead of (\ref{eq:e}) and this bound can be achieved {e.g. in the case if the algorithm always moves in the same direction towards the boundary}. This implies that in order to keep the convergence rate as in the original FW while satisfying $x_t\in S_t(\bar\de)$, due to Fact 2 we have to reduce the uncertainty of the boundaries faster, i.e., we need to take more measurements at each iteration.
\section{\Cref{thm:conv}}\label{proof:thm2}
\begin{proof}
Let 
$\mu_t(\bar \de)$ denote a constant such that $\bar E_t(\bar\de) = \frac{1}{2}\mu_t(\bar\de) \gamma_t C_f$. Then with probability $1-\bar\de$ we have 
$$\la \hat{s},\nabla f(x_t)\ra \leq \min_{s\in D} \la s, \nabla f(x_t)\ra + \frac{1}{2}\mu_t(\bar\de) \gamma_t C_f.$$
For the proof we refer to the following result from \cite{jaggi2013revisiting}. {This result holds in our setting {since} we use the same notions as in \cite{jaggi2013revisiting} of $g_t$ and $s_t$ defined in (\ref{sol_gap}). }
\begin{lemma}{(Lemma 5 \cite{jaggi2013revisiting})}
For a step $x_{t+1}= x_t + \gamma (\hat s-x_t)$ with an arbitrary step-size $\gamma \in [0,1]$, it holds that $$f(x_{t+1}) \leq f(x_t) - \gamma g_t + \frac{\gamma^2}{2}C_f(1+\mu_t),$$ 
{if $\hat s$ is an approximate linear minimizer, i.e. $$\la \hat s, \nabla f(x_t) \ra \leq \min_{\bar s\in D} \la \bar s, \nabla f(x_t)\ra + \frac{1}{2}\mu_t \gamma C_f.$$ }
\end{lemma}

 The step-size of the SFW algorithm is equal to $\gamma_t =  \frac{1}{t+2}$.
Let us define $h_t$  as follows \begin{align*}
h_t &= h(x_t) = f(x_t) - f(x_*).\end{align*}
Then we obtain that
\begin{align*}&h_{t+1} \leq h_t - \gamma_t g_t + \gamma_t^2 \frac{C_f}{2}(1+\mu_t(\bar\de)) \\&\leq h_t - \gamma_t h_t + \gamma_t^2 \frac{C_f}{2}(1+\mu_t(\bar\de)) \\
&= (1-\gamma_t)h_t + \gamma_t^2\frac{C_f}{2}(1+\mu_t(\bar\de)).\end{align*}


If we continue in the same manner, we obtain
\begin{align*}
&h_{t+1}  \leq  \prod_{i=0}^t(1-\gamma_i)h_0 + \sum_{k = 0}^t \gamma_k^2\frac{C_f}{2}(1+\mu_k(\bar\de))\prod_{i = k}^t(1-\gamma_i) \\
 &=  \prod_{i=0}^t\frac{i+1}{i+2}h_0 + \sum_{k = 0}^t \frac{1}{(k+2)^2}\frac{C_f}{2}(1+\mu_k(\bar\de))\prod_{i = k}^t\frac{i}{i+2} \\
&=\frac{1}{t+2}h_0+\sum_{k = 0}^t \frac{1}{(k+2)^2}\frac{C_f (1+\mu_k(\bar\de))}{2}\frac{(t+1)!(k+2)!}{(t+2)!(k+1)!}\\
&=\frac{1}{t+2}\left(h_0+\sum_{k = 0}^t \frac{1}{(k+2)}\frac{C_f (1+\mu_k(\bar\de))}{2}\right).
\end{align*}

 Recall that $\bar E_t(\bar \de)$ denotes the upper bound on $E_t$ with the confidence level $1-\bar \de$. 

Due to \Cref{prop:dfs}, we have $$E_k(\bar \de) = \frac{M C_{\bar \de}}{\sqrt{N_k}} = \frac{M C_{\bar \de}}{\sqrt{C_n}(k+2) \ln (k+2)}.  $$
Hence, we obtain that $$\mu_k(\bar\de) = \frac{2 E_k(\bar \de) (k+2)}{C_f} = \frac{2M C_{\bar \de}}{C_f\sqrt{C_n} \ln (k+2)}.$$

Therefore, we obtain \begin{align*}
h_{t+1} & \leq \frac{h_0+\ln {(t+2)} \frac{C_f}{2} + \sum_{k=0}^t \frac{C_f\mu_k(\bar\de)}{2} }{t+2}=\\
& =  \frac{h_0+\ln (t+2) \frac{C_f}{2} +\ln\ln (t+2)\frac{C'}{2}}{t+2},
\end{align*}
where $C' = \frac{ M C_{\bar\de}}{\sqrt{C_n}}.$
\end{proof}
\paragraph{\Cref{cor:conv}}

\begin{proof}Recall that $$\phi^{-1}(\bar \de) =  \max\left\{
\sqrt{128 d \log N_t \log \left(\frac{N_t^2}{\bar\delta}\right)},
\frac{8}{3}\log \frac{N_t^2}{\bar\delta}\right\}.$$ Hence, $$\phi_{\bar\delta} =  
O\left(\sigma\max\left\{ \sqrt{d} \log t \sqrt{\log \frac{1}{\bar\delta}}, \log t + \log \frac{1}{\bar\delta}\right\}\right).$$
 Recall that the total number of measurements $N_t$  satisfies \begin{align*}&N_t = 2C_{\bar \de}^2\max\left\{
 \frac{4 (\ln \ln T)^2L_{A}^2 }{[\e_0]^2},\frac{1}{(\Gamma_0 +1)^2}\right\}\cdot\\
 &\cdot (t+2)^2(\ln(t+2))^2,\\
  &\text{where } C_{\bar \de} = \frac{2 \phi_{\bar \de} d (\Gamma_0+1) }{\rho_{\min}(D)}\sqrt{\frac{\Gamma_0^2+1}{\omega_0^2} + 1}.\end{align*}
 Hence, we conclude $$N_t  = \tilde O\left(\frac{\phi_{\bar \de}^2d^2}{t^2}\right)= \tilde O\left(\max\left\{\frac{d^3 \ln\frac{1}{\bar \de}}{\epsilon^2}, \frac{d^2 \ln^2\frac{1}{\bar \de}}{\epsilon^2}\right\}\right).$$
\end{proof}

\end{document}